\documentclass[11pt,letterpaper]{amsart}
\usepackage{amsfonts}
\usepackage{amssymb, amsmath, amsthm}
\usepackage[foot]{amsaddr}
\usepackage{enumitem}
\usepackage{tikz}
\usepackage{hyperref}

\theoremstyle{plain}
\newtheorem{thm}{Theorem}[section] 

\newtheorem{lem}[thm]{Lemma}
\newtheorem{prop}[thm]{Proposition}
\newtheorem{cor}[thm]{Corollary}
\newtheorem{rem}[thm]{Remark}

\providecommand{\N}{\mathbb{Z}^{> 0}}
\providecommand{\R}{\mathbb{R}}
  
\providecommand{\Z}{\mathbb{Z}}

\providecommand{\cG}{\mathcal{G}}

\providecommand{\cJ}{\mathcal{J}}

\providecommand{\cNP}{\mathcal{N}\!{\mathcal{P}}}

\providecommand{\eps}{\varepsilon}

\providecommand{\weak}{\rightharpoonup}

\DeclareMathOperator{\spann}{span}
\DeclareMathOperator{\codim}{codim}
\DeclareMathOperator{\essinf}{essinf}

\newcommand{\intd}{\,\mathrm{d}}

\newcommand{\intervalcc}[1]{\mathopen[#1\mathclose]}
\newcommand{\intervalco}[1]{\mathopen[#1\mathclose)}
\newcommand{\intervaloc}[1]{\mathopen(#1\mathclose]}
\newcommand{\intervaloo}[1]{\mathopen(#1\mathclose)}

\title[Normalized solutions of Nehari-Pankov type]{Normalized solutions of Nehari-Pankov type
to mass-supercritical indefinite variational problems}
\date{\today}

\author{Damien Galant\textsuperscript{1,2}}
\author{Tobias Weth\textsuperscript{3}}
\address{\textsuperscript{1}Department of Mathematics, Brown University,
Providence, RI 02912. USA}
\address{\textsuperscript{2}Département de Mathématique,
UMONS, Université de Mons, 7000 Mons. Belgium}
\address{\textsuperscript{3}Institut f\"ur Mathematik, Goethe-Universit\"at Frankfurt, D-60629 Frankfurt am
Main. Germany}
\email{\textsuperscript{1}\href{mailto:damien\_galant@brown.edu}{damien\_galant@brown.edu}}
\email{\textsuperscript{3}\href{mailto:weth@math.uni-frankfurt.de}{weth@math.uni-frankfurt.de}}

\subjclass[2010]{Primary 35J91, secondary 35J35, 35Q55, 35R02, 58E30}

\keywords{Nehari-Pankov set, normalized solutions, indefinite variational problems, spectral gaps, metric graphs}

\begin{document}
\begin{abstract}
We consider abstract nonlinear equations of the form
    $A u = \lambda u + I'(u)$, where $A$ is a self-adjoint
    operator with compact resolvent on a Hilbert space $H$,
    $\lambda \in \R$ is a parameter,
    and $u \mapsto I'(u)$
    is a superlinear term of variational nature. In this abstract setting, we develop a new approach to detect prescribed norm solutions in $H$ which does not rely on any mass-subcriticality assumptions. We then consider various applications of this approach.
    First, we obtain, under general assumptions including the full mass-supercritical parameter regime, the existence of (infinitely many) solutions to a class of nonlinear Schrödinger equations on a compact graph $\cG$ with prescribed arbitrarily large mass, thereby improving previous results which only cover small masses. Moreover, we derive a similar result for a biharmonic Schrödinger equation in the $2$-torus. For a larger class of second order and higher order equations in a bounded domain with Dirichlet boundary conditions, we also show the existence of multiple solutions with prescribed small mass.

    The solutions we obtain are detected as ground states of Nehari-Pankov type for the associated $\lambda$-dependent action functional, where $\lambda$ varies in a spectral gap between sufficiently large eigenvalues of $A$. The key new observation in this abstract framework
    is the fact that the $H$-norms of these $\lambda$-dependent solution families form connected sets even though the solution families themselves may be disconnected. To estimate the size of these connected sets in specific settings, we use Weyl type estimates for the length of spectral gaps, variational characterizations of eigenvalues, bounds for associated eigenfunctions and a bound from analytic number theory.
\end{abstract}

\maketitle

\allowdisplaybreaks

\section{Introduction}

Let $H$ be an infinite dimensional separable real Hilbert space with scalar product $\langle \cdot, \cdot \rangle_H$ and induced norm $\|\cdot\|_H$, and let $A: D(A) \subset H  \to H$ be a nonnegative self-adjoint operator with compact resolvent. We let $E:= D(A^{\frac{1}{2}}) \subset H$, where $A^{\frac{1}{2}}$ is defined in spectral theoretic sense. Then $E$ is a separable
Hilbert space with scalar product
$$
(u,v) \mapsto \langle u,v \rangle_E := \langle u, v \rangle_H +
\langle A^{\frac{1}{2}} u,A^{\frac{1}{2}} v \rangle_H
$$ 
and induced norm $\|\cdot\|_E$. We consider the nonlinear equation
\begin{equation}
    \label{eq:nonlinear-wave-intro}
    A u = \lambda u + I'(u), \qquad u \in E
\end{equation}
with a given (nonlinear) functional $I \in C^1(E,\R)$ satisfying $I(0)=I'(0)=0$. This equation is understood in weak sense as an equation in $E'$, the topological dual of $E$ and it can be written as the Euler-Lagrange equation $\cJ_\lambda'(u)=0$ for the associated $\lambda$-dependent {\em action functional}
$$
\cJ_\lambda \in C^1(E,\R), \qquad \cJ_\lambda(u) = \frac{1}{2}q_\lambda(u)- I(u).
$$
Here and in the following, the symmetric bilinear form $q_\lambda: E \times E \to \R$ is defined by 
 $$
 q_\lambda(u,v):= \langle A^{\frac{1}{2}} u,A^{\frac{1}{2}} v \rangle_H - \lambda \langle u, v \rangle_H
 $$
and we write $q_\lambda(u):= q_\lambda(u,u)$ for $u \in E$.
So, weak solutions of (\ref{eq:nonlinear-wave-intro}) are the critical points of $\cJ_\lambda$. In this paper, we are mainly concerned with the existence of solutions of (\ref{eq:nonlinear-wave-intro}) satisfying the normalization condition
\begin{equation}
  \label{eq:normalization-condition-abstract}
\|u\|_H = \mu   
\end{equation}
for some given $\mu>0$. This constrained problem also has a variational formulation: its solutions are precisely the critical points of the restriction of the {\em energy functional} $\cJ_0$ to the sphere $S_\mu:= \{u \in H\::\: \|u\|_H= \mu\}$, and the frequency parameter $\lambda$ appears as an a priori unknown Lagrange multiplier.

In the present work, however, we do not adopt this point of view. Instead, we develop a new approach related to the action functional which allows to control $\|\cdot\|_H$ and yields additional information on $\lambda$ at the same time. This approach is inspired by the recent paper \cite{DDGS} of De Coster, Dovetta, Serra and the first author, in which ground state solutions of (\ref{eq:nonlinear-wave-intro}) were considered under the assumption that $\lambda$ is below the spectrum $\sigma(A)$ of $A$. One aim of the present work is to remove this restriction, and this leads us to the study of critical points of Nehari-Pankov type of $\cJ_\lambda$ within an indefinite variational setting. These solutions
have received extensive attention in recent years, but no attempt has been made yet to derive normalization conditions for them. 

\medbreak
To show the strength of our new approach, we first discuss an application to nonlinear (time-independent) Schrödinger equations on compact metric graphs, for which we obtain a remarkably general new existence and multiplicity result for solutions with prescribed mass. We recall that a compact metric graph $\cG$ is made of a finite number
of edges, identified to compact intervals of $\R$,
joined together at vertices
(see e.g.\,\cite[Chapter 1]{BerKuc} for an introduction
to analysis on metric graphs). On every edge of the graph, we consider
the differential equation given by
\begin{equation}
\label{graph-ode}
  -u'' = \lambda u + f(u)
  \end{equation}
with a continuous nonlinearity $f$. Moreover, we allow for general vertex conditions
ensuring self-adjointness of the associated Laplacian on $\cG$. Their classification is a classic
result in the theory of quantum graphs,
see e.g.\,\cite[Theorem 1.4.4]{BerKuc}.
Among the possible vertex conditions ensuring self-adjointness,
the most studied
are continuity/Kirchhoff vertex conditions,
in which case one has
\begin{equation*}
    D(A_{\mathrm{Kir}})
    := \biggl\{ u \in C(\cG, \R) \Bigm| 
    u \text{ is $H^2$ edge by edge
    and is s.t.\,for all vertices $v$,}
    \sum_{e \succ v} u_e'(v) = 0 \biggr\}.
\end{equation*}
Here $u_e'(v)$ is the derivative of $u$
along edge $e$ outgoing from the vertex $v$.
We remark that self-adjointness is preserved if the Kirchhoff condition
is replaced by the Dirichlet one at some vertices,
or by the $\delta$-condition
\begin{equation*}
    \sum_{e \succ v} u_e'(v) = \alpha_v u(v)
\end{equation*}
for some real parameters $\alpha_v$.

These conditions have also received a lot of attention
in the recent literature about NLS equations on graphs.
We refer to \cite{KNP}
for an overview of this rapidly growing field.

Our main result in this graph setting is the following.
\begin{thm}
    \label{graph_mult}
    Let $\cG$ be a compact metric graph
    and $(D(A), A)$ be a self-adjoint realisation
    of the Laplacian on $\cG$
    (acting as the second derivative on every edge).
    Let $f \in \mathcal{C}(\R, \R)$ satisfy the following two assumptions:
    \begin{itemize}
    \item[(f1)]\label{def:f1} The function $s \mapsto \frac{f(s)}{|s|}$ is strictly increasing on $\R \setminus \{0\}$ with
      $$
      \lim_{s \to 0}\frac{f(s)}{|s|}= 0\qquad \text{and}\qquad \lim_{s \to \pm \infty}\frac{f(s)}{|s|} =  \pm \infty.
      $$
    \item[(f2)]\label{def:f2}Defining
    $F(s) := \int_0^s f(t) \intd t$,
    there exists $s_0 > 0$ and $\kappa_0>1$ with 
    \begin{equation}
        \label{hyp_homogeneity}
        F\Bigl(\frac{s}{\kappa_0}\Bigr)
        \le F(-s) \le F(\kappa_0 s)
        \qquad \text{for all $s \ge s_0$.}
      \end{equation}
    \end{itemize}
    Then, for any $\mu > 0$, the nonlinear Schrödinger equation
    $-u'' = \lambda u + f(u)$
    coupled with vertex conditions from $D(A)$
    has infinitely many normalized solutions of mass $\mu$,
    associated to a sequence of parameters
    $\lambda$ converging to $+\infty$.
    Moreover, if $A$ admits one eigenfunction
    which is positive almost everywhere in $\cG$,
    then these solutions are sign-changing for $\lambda$ large.
\end{thm}

\begin{rem}{\rm The assumption (f1) is a natural superlinearity condition.  The condition (f2) can be seen as a mild form of approximate homogeneity. Since, by (f1), the function $F$ is increasing on $\intervalco{0, +\infty}$ and decreasing on $(-\infty,0]$,
    one can always replace $\kappa_0$ by a bigger positive constant in \eqref{hyp_homogeneity}.
  }
\end{rem}

Under assumptions of Theorem~\ref{graph_mult}, even the mere existence of one (non-constant) normalized solution is new in the case of an arbitrary given mass $\mu>0$. In comparison with the existing literature, it is remarkable that Theorem~\ref{graph_mult} neither requires oddness of the nonlinearity $f$, nor does it require any restriction on the growth of $f$ relative to the associated mass-critical exponent.

To discuss this in more detail, we note that the special case where $f(u) = |u|^{p-2}u$ for some $p > 2$ has received particular
interest in the literature. When $2 < p < 6$ and in the case of continuity/Kirchhoff
conditions, an infinite multiplicity result
(without information on $\lambda$ or on the sign
of solutions) was obtained
by Dovetta using Lusternik-Schnirelmann theory
on the energy functional constrained
to the sphere of constant mass \cite{Dov18}.

In the mass-supercritical case $p>6$, existence of non-constant normalized
solutions is more delicate since the energy functional
is unbounded from below on the constraint.
A first result on the existence of nonconstant solutions
for small masses was obtained by Chang, Jeanjean and Soave
\cite{CJS}. As mentioned above, the problem of existence in the case of arbitrary masses remained open on compact graphs for $p>6$, while it could be solved in a different setting of noncompact graphs and ``localized
nonlinearities'' in \cite{BCJS_nonlin}. Moreover, an infinite multiplicity result was obtained
in this case by Carrillo, Jeanjean, Troestler and the first
author in \cite{CGJT}, exploiting a new abstract result
about the existence of bounded Palais-Smale sequences with
Morse index type information recently developed
by Borthwick, Chang, Jeanjean and Soave \cite{BCJS_TAMS}.
Those results and the underlying methods depend in a crucial way on the presence of half-lines in the graph without a nonlinearity acting on them. In this way, they exploit a suitable tradeoff between compactness and non-compactness, and thus they do not carry over to the case of compact graphs.

Very recently, in \cite[Theorem 1.3]{JeSo}, Jeanjean and Song proved a multiplicity result for compact graphs $\cG$
and small masses with continuity-Kirchhoff conditions. More precisely, by combining flow invariance techniques with delicate estimates, they show that for every $n \in \N$ there exists $\mu_n>0$ with the property that, for every $\mu \in (0,\mu_n)$, the 
NLS equation $-u'' = \lambda u + |u|^{p-2}u$ admits at least $n$ pairs $\pm u$ of sign changing solutions with fixed mass $\mu$. 
Using our new ``frequency to masses'' tool (Theorem~\ref{M-n-char} below), we are able to improve
this multiplicity result to the existence of infinitely many solutions, for any mass. Moreover, we can extend this multiplicity result beyond the class of odd nonlinearities to a setting where, in particular, Lusternik-Schnirelman theory does not apply. Let us also point out that
we can deal with the most general self-adjoint
vertex conditions, also generalizing the previous results.
In the case of continuity-Kirchhoff conditions,
the first eigenfunction of $A$ is positive in $\cG$
(see e.g.\,\cite[Theorem 4.12]{Kur}),
so that we also obtain the existence
of sign-changing solutions.
Finally, as already pointed out, we do not need to make
any case distinction depending on whether $p < 6$ or $6 < p$.

\medbreak
For our next result, let $\Omega  \subset \R^N$ be
a smooth bounded domain and $m \ge 1$ be an integer. Moreover, we let
    \begin{equation*}
        2^*_{N, m} := \begin{cases}
            \frac{2N}{N-2m}  &\text{if $N > 2m$},\\
            +\infty         &\text{otherwise}
        \end{cases}
    \end{equation*}
    denote the usual $m$-th order Sobolev critical exponent, and we consider the polyharmonic equation
    \begin{equation}
      \label{polyharmonic}
                (-\Delta)^m u = \lambda u + f(x, u)\qquad \text{in } \Omega
\end{equation}
with a nonlinearity $f \in C(\Omega \times \R,\R)$ satisfying
\begin{itemize}
\item[$(FP1)$] $|f(x,t)| \le a(1+|t|^{p-1})$ for some $a>0$ and $p\in(2,2^*_{N,m})$.
\item[$(FP2)$]  $f(x,t)=o(t)$ uniformly in $x$ as $|t| \to 0$.
\item[$(FP3)$] $F(x,t)/t^2 \to\infty$ uniformly in $x \in \Omega$ as $|t|\to\infty$, where $F(x,t) = \int_0^t f(x,s) \intd s$.
\item[$(FP4)$] $t \mapsto f(x,t)/|t|$
is strictly increasing on $\R \setminus \{0\}$
for all $x \in \Omega$.
\end{itemize}
We consider (\ref{polyharmonic}) both with homogeneous Dirichlet boundary conditions
\begin{equation}
  \label{eq:Dirichlet-boundary}
u = \partial_\nu u = \dots = \partial_\nu^{m-1} u = 0\qquad \text{on } \partial \Omega
  \end{equation}
and homogeneous Navier boundary conditions
\begin{equation}
  \label{eq:Navier-boundary}
            u = \Delta u = \cdots = \Delta^{m-1} u = 0
\qquad \text{on } \partial \Omega.
\end{equation}
Our main result in this setting is a straightforward consequence of an abstract multiplicity
result given in Theorem~\ref{abstract_mult} below, and it reads as follows.

\begin{thm}
  \label{theorem-polyharmonic}
Suppose that (FP1)--(FP4) are satisfied. Then, for any integer $k \ge 1$,
    there exists $\mu_k > 0$ such that for all
    $\mu \in \intervaloo{0, \mu_k}$,
    problems (\ref{polyharmonic})--(\ref{eq:Dirichlet-boundary}) and (\ref{polyharmonic})--(\ref{eq:Navier-boundary})
    both admit $k$ distinct  
    solutions $(u,\lambda)$
    with mass $\|u\|_{L^2(\Omega)}^2 = \mu$. 
    Moreover, if $m=1$, these solutions are sign-changing.
  \end{thm}

We emphasize that this result is new even in the classical second order case $m=1$, and it does not require oddness of the nonlinearity $f$ in $t$. In the special case of homogeneous nonlinearities of the form $f(x,t)= |t|^{p-2}t$, Theorem~\ref{theorem-polyharmonic} is the analogue of the recent multiplicity result of Jeanjean and Song in \cite[Theorem 1.3]{JeSo} for the case of compact metric graphs. Our method is completely different from the approach in \cite{JeSo}, which, according to remarks in \cite[Introduction]{JeSo}, also yields multiplicity results for NLS-equations on bounded domains in $\R^N$.
  
We also point out that the conclusion of Theorem~\ref{theorem-polyharmonic} is weaker than the one of Theorem~\ref{graph_mult}, since only small masses are considered. We conjecture that at least in the case $N=2$, $m>1$ the multiplicity result of Theorem~\ref{theorem-polyharmonic} extends to arbitrary positive masses $\mu$. To support this conjecture, we present a further result for the closely related problem of finding normalized solutions of biharmonic semilinear equations in the $2$-torus. 

\begin{thm}
    \label{torus_mult}
    Let $\Omega := S^1 \times S^1$ be the flat 2-torus,
    let $p>2$, and let $r \in L^\infty(\Omega)$ satisfy
    $\underset{x \in \Omega}{\essinf}\,r(x)>0$.
    Then, for any $\mu>0$, the nonlinear equation
    \begin{equation}
      \label{torus-eq}
        \Delta^2 u
        = \lambda u + r(x)|u|^{p-2}u
        \qquad \text{in $\Omega$}
    \end{equation}
    admits infinitely many normalized sign-changing solutions with mass $\|u\|_{L^2(\Omega)}^2 = \mu$ associated to a sequence of parameters $\lambda$ converging to $+\infty$.
  \end{thm}
  
  \begin{rem}
 {\rm
(i) In the mass supercritical regime $p>6$ and for a given positive mass $\mu>0$, even the mere existence
of one normalized solution of (\ref{torus-eq}) is new. If the weight function $r$ is constant,
then the symmetries of the torus can be used
to produce multiple solutions with large mass
via a reflection procedure (see \cite[Theorem 1.13 and Remark 1.1]{pierotti-verzini} for related results in bounded domains). This is not possible when
$r$ does not have any symmetries.\\
(ii) Our proof of Theorem~\ref{torus_mult} relies on the explicit knowledge of eigenvalues and eigenfunctions of the Laplacian on the flat torus $S^1 \times S^1$, and it directly extends to the equation
\begin{equation}
  \label{eq:sqare-probl}
        \Delta^2 u
        = \lambda u + r(x)|u|^{p-2}u
        \qquad \text{in $\Omega= (0,\pi) \times (0,\pi)$}
\end{equation}
        subject to Navier boundary conditions
\begin{equation}
  \label{eq:sqare-probl-bc}
        u = \Delta u = 0 \qquad \text{on $\partial \Omega$.}        
\end{equation}
Here, again we assume that $p>2$ and $r \in L^\infty(\Omega)$ satisfies
    $\underset{x \in \Omega}{\essinf}\,r(x)>0$.
    Hence, for any $\mu>0$, problem (\ref{eq:sqare-probl}), (\ref{eq:sqare-probl-bc}) also admits infinitely many normalized sign-changing solutions with mass $\|u\|_{L^2(\Omega)}^2 = \mu$ associated to a sequence of parameters $\lambda$ converging to $+\infty$.}
\end{rem}

\medbreak
In the following, we present our abstract approach and explain how the results above follow from it. We continue to use the notations $H,E, A, I,q_\lambda$ and $\cJ_\lambda$ introduced in the beginning of this section.
We assume the functional $I$ satisfies
the following abstract properties.
\begin{itemize}
\item[(I1)] $I$ is strongly continuous, i.e. $u_n \weak u$ in $E$ implies $I(u_n) \to I(u)$.
\item[(I2)]  $I(0)=0$, and $I(u)>0$ for $u \in E \setminus \{0\}$. 
\item[(I3)] For every $u \in E \setminus \{0\}$, the function
       $t \mapsto \frac{I(tu)}{t^2}$ is nondecreasing on $(0,\infty)$ with
       \begin{equation}
         \label{eq:I-limits}
        \lim_{t \to 0^+} \frac{I(tu)}{t^2} = 0 \qquad \text{and}\qquad   \lim_{t \to +\infty} \frac{I(tu)}{t^2} = +\infty. 
       \end{equation}                  
\item[(I4)] For $u \in E \setminus \{0\}$ and $v \in E$ we have 
   \begin{equation}\label{eq:criterionB2I}
     I(u)-I((1+s)u+v)+I'(u)[s(s/2+1)u+(1+s)v]\leq 0\quad \text{for all $v\in E$, $s \ge -1$,}
   \end{equation}
   with strict inequality if $v=0$ and $s \ne 0$. 
     \end{itemize}

\begin{rem}{\rm 
    \label{I4_practice}
    With regard to (I4), we note that
    a function $F \in C^1(\R)$
    such that $\frac{F'(t)}{|t|}$ is strictly increasing
    on $\R \setminus \{0\}$ with
    $\lim \limits_{t \to 0} \frac{F'(t)}{|t|} = 0$ satisfies 
    \begin{equation}
    \label{eq:szulkin-weth-lemma}
    F(u)-F((1+s)u+v) + F'(u)[s(s/2+1)u+(1+s)v] \leq 0
    \qquad\;\text{for all } u,v\in\R,s\geq -1
    \end{equation}
    with strict inequality if $v=0$ and $s \ne 0$,
    see \cite[Lemma~38]{SzuWet}.
    This is very useful in applications.
 }   
\end{rem}

The assumptions on $I$ allow to formulate, for fixed $\lambda$,  a very useful variational principle for least action nonzero critical points of the functional $\cJ_\lambda$, which are also called ground state solutions of (\ref{eq:nonlinear-wave-intro}). To introduce this principle, we need some more notations. Since we assume that $A$ is nonnegative and has compact resolvent, the spectrum of $A$ consists of an unbounded
sequence of eigenvalues 
$$
0 \le \lambda_1 \le \lambda_2 \le \dots \le \lambda_n \le \dots \to +\infty
$$
(counted with multiplicity), and there exists an orthonormal basis of $H$ of associated eigenfunctions
$\zeta_n\in D(A)$, $n \in \N$. In the following, we fix $n \in \N$ and we consider values of $\lambda$ satisfying 
\begin{equation}
  \label{eq:lambda-main-spectral-assumption}
\lambda_n < \lambda < \lambda_{n+1},
\end{equation}
so that $\lambda$ lies in a gap between
two consecutive eigenvalues. Moreover, we let
$$
E_n:= \spann \{\zeta_1,\dots,\zeta_n\},
$$
and for $u \in E \setminus E_n$ we consider the subspace
\begin{equation}
  \label{eq:def-F-u}
F_n(u):= E_n \oplus \R u  
\end{equation}
and the half space\footnote{Here and in the following,
we let $\R^+ := (0,\infty)$ as usual.}
\begin{equation}
  \label{eq:def-F-u+}
  F_n(u)^+:=E_n \oplus \R^+ u   \:\subset F_n(u).
\end{equation}
Moreover, we let $\cNP_\lambda$ denote the associated {\em Nehari-Pankov set}\footnote{This set is also called Nehari-Pankov manifold or generalized Nehari manifold, but additional restrictions are needed to guarantee that $\cNP$ is a $C^1$-manifold, see \cite{pankov-2005}. We will not use such a manifold structure in this paper.},
defined by  
\begin{equation*}
    \cNP_\lambda:= \Bigl\{
        u \in E \setminus E_n
        \bigm| \cJ_\lambda'(u)\Big|_{F_n(u)} = 0
    \Bigr\}.
\end{equation*}
We then have the following result. 
\begin{thm}
    \label{main-abstract-result-existence-lambda}
Under the assumptions above, the Nehari-Pankov set $\cNP_\lambda$ is non-empty and contains all nonzero critical points of $\cJ_\lambda$.
    Moreover, we have 
    \begin{equation}
        \label{eq:inf-sup-characterization}
        c_\lambda:= \inf_{\cNP_\lambda} \cJ_\lambda = \inf_{u \in E \setminus E_n} \sup_{F_n(u)^+} \cJ_\lambda \:>\:0.
    \end{equation}
    Moreover, the first infimum is attained, and every minimizer
    of $\cJ_\lambda$ on $\cNP_\lambda$ is a critical point
    of $\cJ_\lambda$. Thus, all minimizers
    of $\cJ_\lambda$ on $\cNP_\lambda$ are ground state solutions of \eqref{eq:nonlinear-wave-intro}.
\end{thm}

This result is essentially known (see \cite{szulkin-weth-2009,SzuWet} and \cite[Section 4]{bartsch-mederski-2017-JFPT}), but we could not find it in this precise form in the literature in the framework given by our assumptions. We note that -- in the context of different semilinear variational problems -- the idea of using the Nehari-Pankov set as a natural constraint goes back to \cite{pankov-2005, ramos-tavares-2008} and has then been developed further in more general settings,
see e.g.\,\cite{szulkin-weth-2009,SzuWet} and \cite{de-Paiva-Kryszewski-Szulkin-2017,bartsch-mederski-2015,bartsch-mederski-2017-JFA,mandel-weth-2025}.

The key additional abstract result of the present paper is the following theorem, which states that the masses of ground state solutions of \eqref{eq:nonlinear-wave-intro} form a connected set if $\lambda$ is varied in the spectral gap given by (\ref{eq:lambda-main-spectral-assumption}). 
\begin{thm}
    \label{M-n-char}
    For $\lambda \in (\lambda_n, \lambda_{n+1})$, let 
    \begin{equation*}
        Q_\lambda
        := \bigl\{u \in \cNP_\lambda \mid \cJ_\lambda(u)
        = c_\lambda\bigr\}  
    \end{equation*}
    and
    \begin{equation*}
        M_n:= \biggl\{\frac{\|u\|_H^2}{2} \Bigm|
        \text{$u \in Q_\lambda$ for some $\lambda \in 
            (\lambda_n, \lambda_{n+1})$}\biggr\}.
    \end{equation*}
    Then,
    \begin{equation}
        \label{eq:M-n-char}
        (0,g_n) \subseteq M_n \subseteq (0,g_n] \qquad \text{for some $g_n>0$.}  
    \end{equation}
\end{thm}

We emphasize that the connectedness of the set $M_n$ stated in Theorem~\ref{M-n-char} does not follow from a standard continuity argument, since we cannot expect connectedness of the set of associated ground state solutions in this general variational framework. Instead, the connectedness of $M_n$ follows by an inspection of the special saddle point nature of the Nehari-Pankov sets $\cNP_\lambda$, while it is not a sole consequence of the continuity of ground state levels $c_\lambda$.

\medbreak
Clearly, to deduce existence and multiplicity results of normalized solutions satisfying (\ref{eq:normalization-condition-abstract}), it is important to know for which $n \in \N$ the inequality $\mu \le g_n$ holds. Our proof of Theorem~\ref{M-n-char} shows that 
\begin{equation}
  \label{eq:g-n-characterization}
    g_n:= \sup_{\lambda \in (\lambda_n, \lambda_{n+1})}g(\lambda),
\end{equation}
    where $g:(\lambda_n, \lambda_{n+1}) \to \R$ is defined by
    $g(\lambda) :=  \liminf \limits_{\eps \to 0^+}\frac{c_{\lambda}-c_{\lambda+\eps}}{\eps}.$ For the special case $\lambda < \lambda_1$, a similar result was recently obtained 
by De Coster, Dovetta, Serra and the first author \cite{DDGS}. In this case, ground state solutions of (\ref{eq:nonlinear-wave-intro}) are minimizers of $\cJ_\lambda$ on the classical Nehari set. Theorem~\ref{M-n-char} is inspired by this previous result, but the proof is more involved in the more general setting which exploits the Nehari-Pankov variational structure instead of the pure Nehari one.

\medbreak
The following is an immediate corollary of Theorem~\ref{M-n-char}. 
\begin{thm}
    \label{abstract_mult}
    Let $k$ be a positive integer, and let $J_1,\dots,J_k$
    be disjoint spectral gaps\footnote{We call
    an interval $J \subset \R$
    a \emph{spectral gap of $A$}
    if $J = (\lambda_n,\lambda_{n+1})$ for some $n \in \N$.}
    of $A$.
    Then, there exists $\mu_k > 0$
    such that, for all $\mu \in (0, \mu_k)$,
    there exist $\Lambda_1, \dotsc, \Lambda_k \in \R$ with $\Lambda_i \in J_i$ for $i = 1,\dots,k$
    and $u_1, \dotsc, u_k \in E$ with
    \begin{equation*}
        Au_i = \Lambda_i u_i + I'(u_i), \qquad \text{and}\qquad \| u_i \|_H = \mu.
    \end{equation*}
    In particular, (\ref{eq:nonlinear-wave-intro}) has at least
    $k$ different solutions $(u,\lambda)$ satisfying (\ref{eq:normalization-condition-abstract}). 
\end{thm}

Theorem~\ref{theorem-polyharmonic} follows in a rather direct way from Theorem~\ref{abstract_mult}, one merely needs to verify that the assumptions of Theorem~\ref{theorem-polyharmonic} fit into the present abstract framework.  

\medbreak
Without additional knowledge on the interval lengths $g_n$  in \eqref{eq:M-n-char}, Theorem~\ref{abstract_mult} can only guarantee the existence of solutions with small mass $\mu$. To obtain solutions with large mass as in Theorems~\ref{graph_mult} and \ref{torus_mult}, we need effective lower bounds for $g_n$. As far as we can see, the characterization given in (\ref{eq:g-n-characterization}) does not give rise to useful lower bounds, so we follow a different idea. Our strategy is to use the length of the associated spectral gap $(\lambda_{n},\lambda_{n+1})$ together with variational characterizations of eigenvalues and bounds for associated eigenfunctions. Moreover, we use Weyl type estimates to bound the length of suitable spectral gaps from below. 

\medbreak
In particular, we are able to prove
Theorem~\ref{graph_mult} by exploiting the presence
of large spectral gaps for Laplacians on compact
metric graphs (a consequence of Weyl's law,
which holds for all self-adjoint realizations
of the Laplacian on graphs as shown in \cite{BolEnd}). Classical spectral theory then shows 
that for parameters $\lambda$ at the bottom of a large spectral gap, solutions of (\ref{graph-ode}) have a large
$L^\infty$-norm. We then conclude using an ODE argument
providing a suitable ``$L^\infty$ to $L^2$'' estimate.

\medbreak
In the case of Theorem~\ref{torus_mult}, we use instead the explicit structure of the set of eigenvalues and eigenfunctions of the Laplacian on the flat torus $\Omega = S^1 \times S^1$ together with a classical bound from analytic number theory for the multiplicities of these eigenvalues. 

\medbreak
We point out again that the methods developed in the present paper
do not rely on any assumptions related to mass-critical exponents. This makes our method particularly well-suited to obtain
new results in the usually harder mass-supercritical case
(as is the case in \cite{DDGS} for minimizers
on the Nehari manifold).  

\medbreak
The paper is organized as follows. In Section~\ref{sec:nehari-pankov}, we further develop the abstract framework of the indefinite variational problem associated with (\ref{eq:nonlinear-wave-intro}) and prove Theorem~\ref{main-abstract-result-existence-lambda}. In Section~\ref{sec:lambda-depend-minim}, we then study the $\lambda$-dependence
of the Nehari-Pankov solutions obtained in Theorem~\ref{main-abstract-result-existence-lambda}, and we complete the proof of Theorem~\ref{M-n-char}. In the end of Section~\ref{sec:lambda-depend-minim}, we then provide useful estimates for the action of Nehari-Pankov solutions under additional assumptions on the functional $I$ and the operator $A$. In Section~\ref{sec:furth-estim-case}, we then specialize to the case $H = L^2(\Omega)$ on a compact metric measure space and derive explicit lower bounds on the largest mass attainable by a Nehari-Pankov ground state within a given spectral gap, which underlie the large-mass results in Theorems~\ref{graph_mult} and~\ref{torus_mult}. Section~\ref{sec:nonl-schr-equat} is devoted to the nonlinear Schrödinger equation on compact metric graphs and to the proof of Theorem~\ref{graph_mult}. Finally, in Section~\ref{sec:nonl-biharm-equat}, we study the equation~(\ref{torus-eq}) on the flat $2$-torus, and we complete the proof of Theorem~\ref{torus_mult}.

\section{Existence of Nehari-Pankov solutions
for fixed $\lambda$}
\label{sec:nehari-pankov}
This section is devoted to the proof
of Theorem~\ref{main-abstract-result-existence-lambda}. We continue to use the notations $H$, $E$, $A$, $I$, $q_\lambda$, $\cJ_\lambda$, $E_n$, $F_n(u)$, $F_n(u)^+$ and $\cNP_\lambda$ introduced in the introduction, and we assume throughout this section that assumptions $(I1)$--$(I4)$ are satisfied.

We first note that, taking $v=0$ and $s=-1$ in $(I4)$, we deduce that for all $u \in E \setminus \{0\}$ we have
\begin{equation}
  \label{conseq_I4}
I(u) < \frac{1}{2}I'(u)u.  
\end{equation}
Next, arguing as in \cite[Chapter~4]{SzuWet}, 
we show that $\cNP_\lambda$ contains all
nonzero critical points of $\cJ_\lambda$.
Let $u \in E \setminus \{ 0 \}$ be a critical
point of $\cJ_\lambda$. Then, according to~\eqref{conseq_I4},
\begin{equation*}
    \cJ_\lambda(u)
    = \cJ_\lambda(u) - \frac12 \cJ_\lambda'(u)[u]
    = \frac12 I'(u)[u] - I(u)
    > 0.
\end{equation*}
Moreover, one has
\begin{equation}
    \label{E_n-nonpositive}
    \sup_{v \in E_n} \cJ_\lambda(v)
    \le \sup_{v \in E_n} \frac{q_\lambda(v)}{2}
    \le 0.
\end{equation}
This implies that $u$ does not belong to $E_n$,
so that $u$ belongs to $\cNP_\lambda$ as claimed.

In the following, it is useful to consider also the orthogonal complement $E_n^\perp \subset E$ of the subspace $E_n$, given by
$$
E_n^\perp := \{w \in E \mid \langle w,v \rangle_E = 0 \text{ for all $v \in E_n$}\} = \{w \in E \mid \langle w,v \rangle_H = 0 \text{ for all $v \in E_n$}\}.
$$
Here, the second equality follows from the fact that
$$
\langle w, \zeta_k \rangle_E = (1+\lambda_k)\langle w, \zeta_k \rangle_H
\qquad \text{for every $k \in \N$, $w \in E$.}
$$
Moreover, every $u \in E \setminus E_n$ can be uniquely written as
$$
u = s w + v \qquad \text{with $s>0$, $v \in E_n$ and $w \in E_n^\perp$ satisfying $q_\lambda(w)=1$,}
$$
and this representation yields (remarking that $w$ belongs to $F_n(u)^+$)
$$
\sup_{F_n(u)^+}\cJ_\lambda
\ge \cJ_\lambda(tw)= \frac{t^2}{2}q_\lambda(w) - I(tw)
= t^2\Bigl(\frac12 - \frac{I(tw)}{t^2}\Bigr)
\qquad \text{for every $t>0$.}
$$
Hence the first limit in (\ref{eq:I-limits}) implies that
\begin{equation}
\label{sup-positive}
0 < \sup_{F_n(u)^+}\cJ_\lambda
= \sup_{\overline{F_n(u)^+}}\cJ_\lambda.
\end{equation}
Here, the second equality follows since
the relative boundary of $F_n(u)^+$ in $F_n(u)$
is the subspace $E_n$ and since
$\cJ_\lambda \le 0$ on $E_n$ according
to \eqref{E_n-nonpositive}. Moreover, we point out that
\begin{equation}
  \label{eq:universal-prop-F(u)}
F_n(u)^+ = F_n(\tilde u)^+ \qquad \text{for every $u \in E \setminus E_n$ and every $\tilde u \in F_n(u)^+$.}   
\end{equation}
We shall need a boundedness property of the set where $\cJ_\lambda$ takes positive values in $F_n(u)^+$. We state this property directly in a form which allows for variation of $\lambda$, since this will be used in the next section.

 \begin{lem}\label{B-R-Lemma}
   Let $I \subset (\lambda_n,\lambda_{n+1})$ be a compact set, and let $u \in E \setminus E_n$. Then there exists $R>0$ with 
$$
\cJ_\lambda \le 0 \qquad \text{on $F_n(u)^+ \setminus B_R(0)\quad$ for all $\lambda \in I$,}
$$
where the ball $B_R(0)$ is taken in the norm $\|\cdot\|_E$.
\end{lem}

\begin{proof}
  We suppose by contradiction that this is false.
  Then there exist sequences $(\lambda^k)^k$
  in $I$ and $(u_k)_k$ in $F_n(u)^+$ with $\|u_k\|_E \to \infty$ and
  \begin{equation}
    \label{eq:contradiction-assumption}
  \cJ_{\lambda^k}(u_k) > 0 \qquad \text{for all $k \in \N$.}
  \end{equation}
  Since the norms $\|\cdot\|_E$ and $\|\cdot\|_H$ are equivalent on the finite-dimensional space $F_n(u)$, we then also have
  $t_k:= \|u_k\|_H \to \infty$ as $k \to \infty$.

  Since the set $S := \{v \in F_n(u) \mid \|v\|_H = 1\}$ is compact, we may pass to a subsequence with $\lambda^k \to \lambda \in I$ and
  $$
  w_k:= \frac{u_k}{t_k} \to w \qquad \text{in $E$ as $k \to \infty$,}
  $$
  where $w \in S$. Moreover, by the second limit in (\ref{eq:I-limits}), there exists $t_*>0$ with
  \begin{equation}
    \label{eq:t-star-w-ineq}
  \frac{I(t_* w)}{t_*^2} > \frac{q_{\lambda_n}(w)}{2}.
  \end{equation}
  Since $\lambda^k > \lambda_n$ for all $k$
  and since $t_k \ge t_*$ for large $k$, it follows that
\begin{align*}
\cJ_{\lambda^k}(u_k) &= \frac{1}{2}q_{\lambda^k}(u_k)- I(u_k) \le \frac{1}{2}q_{\lambda_n}(u_k)- I(u_k)= t_k^2\Bigl(\frac{q_{\lambda_n}(w_k)}{2}- \frac{I(t_k w_k)}{t_k^2}\Bigr)\\
& \le t_k^2 \Bigl(\frac{q_{\lambda_n}(w_k)}{2}-\frac{I(t_* w_k)}{t_*^2}\Bigr)=
t_k^2 \Bigl(\frac{q_{\lambda_n}(w)}{2}
- \frac{I(t_* w)}{t_*^2} + o(1)\Bigr)
\end{align*}
as $k \to \infty$. Hence,
   $$
   \cJ_{\lambda^k}(u_k) \le 0 \qquad \text{for large $k$,}
   $$
   contrary to (\ref{eq:contradiction-assumption}).
   The claim thus follows.
\end{proof}

\begin{prop} \label{prop:criterionB2I}
Let $u \in \cNP_\lambda$. Then we have
       \begin{equation}
      \label{unique-max-statement}
        \cJ_\lambda(u)>\cJ_\lambda(w) \qquad \text{for all $w \in F_n(u)^+ \setminus \{u\}$}.
    \end{equation}
Namely, $u$ is the unique (and strict) global maximizer of the restriction of $\cJ_\lambda$ to $F_n(u)^+$.
  \end{prop}

  \begin{proof}
  Let $w \in F_n(u)^+$, which can be written as $w=(1+s)u+v$ with some $v \in E_n$
    and some $s > -1$. Since also $s(s/2+1)u+(1+s)v \in F_n(u)$, it follows from $u \in \cNP_\lambda$ that
    \begin{equation*}
        \cJ_\lambda'(u)[s(s/2+1)u+(1+s)v]=0
    \end{equation*}
    and therefore 
  \begin{align*}
    \cJ_\lambda(w)-\cJ_\lambda(u)
    &= \cJ_\lambda((1+s)u+v)-\cJ_\lambda(u) -  \cJ_\lambda'(u)[s(s/2+1)u+(1+s)v] \\
    &= \frac{1}{2}q_\lambda(v) + \Big[\,I(u)-I((1+s)u+v) +  I'(u)[s(s/2+1)u+(1+s)v]\,\Big]\\ 
    &\le \frac{\lambda_n-\lambda}{2}\|v\|_H^2 + \Big[\,I(u)-I((1+s)u+v) +  I'(u)[s(s/2+1)u+(1+s)v]\,\Big]. 
  \end{align*}
   By $(I4)$ and since $\lambda_n < \lambda$, both terms on the right-hand side are nonpositive, which yields $\cJ_\lambda(w)\leq\cJ_\lambda(u)$.   
   Moreover, if $\cJ_\lambda(w)=\cJ_\lambda(u)$ holds,
   then we necessarily have $\|v\|_H=0$ and thus $v=0$,
   so our assumption about the strict inequality in $(I4)$
   implies $s =0$, i.e.\,$w=u$.
   So $u$ is the unique global maximizer, as claimed.
 \end{proof}

 \begin{cor}
   \label{existence-and-uniqueness}
   For every $\tilde u\in E \setminus E_n$,
   there exists a unique maximizer $u$
   of the restriction of $\cJ_\lambda$ to $F_n(\tilde u)^+$,
   and $u$ belongs to $\cNP_\lambda$.
   In particular, this implies that
   $$
   \inf_{\cNP_\lambda} \cJ_\lambda
   = \inf_{\tilde u \in E \setminus E_n}
   \sup_{F_n(\tilde u)^+} \cJ_\lambda.
   $$
 \end{cor}

 \begin{proof}
   Let $\tilde u\in E \setminus E_n$, let $R>0$ be given by Lemma~\ref{B-R-Lemma} for the compact set $I = \{\lambda\}$, and let $C \subset F_n(\tilde u)$ denote the closure of $F_n(\tilde u)^+ \cap B_R$ in $F_n(\tilde u)$. Then we have
   $$
   0 < \sup_{F_n(\tilde u)^+}\cJ_\lambda = \sup_{F_n(\tilde u)^+\cap B_R}\cJ_\lambda \le \sup_{C}\cJ_\lambda, 
   $$
by (\ref{sup-positive}) and Lemma~\ref{B-R-Lemma}. Since the set $C$ is compact, $\cJ_\lambda$ attains its maximum on $C$ at a point $u$. Since the relative boundary of $C$ in $F_n(\tilde u)$ is contained in $E_n \cup (F_n(\tilde u)^+ \cap \partial B_R)$ and
   $\cJ_\lambda \le 0$ on this set by (\ref{E_n-nonpositive}) and Lemma~\ref{B-R-Lemma}, it follows that $u$ is contained in the relative interior of $F_n(\tilde u)^+\cap B_R$ in $F_n(\tilde u)$. Hence $u \in F_n(\tilde u)^+$, and $u$ is a maximizer of the restriction of $\cJ_\lambda$ to $F_n(\tilde u)^+$. We therefore have $\cJ_\lambda'(u) \Big|_{F_n(\tilde u)} \equiv 0$, i.e.\,$u \in \cNP_\lambda$.
Since $F_n(\tilde u)^+= F_n(u)^+$
by (\ref{eq:universal-prop-F(u)}),
the uniqueness of $u$ now follows
from Proposition~\ref{prop:criterionB2I}.
 \end{proof}

Next we state a weak compactness property in a general
$\lambda$-dependent form which will be needed
in the next section. 
\begin{lem}
  \label{general-compactness-cond}
  Let $(\lambda^k)^k \subset (\lambda_n,\lambda_{n+1})$ and $(w_k)_k \subset E_n^\perp$ be sequences with $\|w_k\|_E=1$ for all $k$, $\lambda^k \to \lambda \in (\lambda_n,\lambda_{n+1})$ as $k \to \infty$ and
  \begin{equation}
    \label{eq:bounded-action-assumption}
   \sup_{k \in \N} \sup_{F_n(w_k)^+}\cJ_{\lambda^k} < \infty. 
  \end{equation}
Then, passing to a subsequence, we have $w_k \weak w$ in $E$ as $k \to \infty$ for some
$w \in E_n^\perp \setminus \{ 0 \}$, and 
  $$
  \sup_{F_n(w)^+}\cJ_\lambda  \le \liminf_{k \to \infty}\sup_{F_n(w_k)^+}\cJ_{\lambda^k}.
  $$
\end{lem}

\begin{proof}
  Since $(w_k)_k$ is bounded in $E$, we may pass to a subsequence with $w_k \weak w$ in $E$ for some $w \in E_n^\perp$. Suppose by contradiction that $w =0$, which implies that for every $t>0$ we have $t w_k \weak 0$ as $k \to \infty$. Therefore, by (I1) and (I2), 
  \begin{equation*}
      \sup_{F_n(w_k)^+}\cJ_{\lambda^k} \ge \cJ_{\lambda^k}(tw_k) = \frac{t^2}{2}q_{\lambda^k}(w_k)- I(t w_k) = \frac{t^2}{2}q_{\lambda}(w_k)+(\lambda^k - \lambda)\|w_k\|_H^2 + o(1) = \frac{t^2}{2}q_{\lambda}(w_k)+ o(1) 
  \end{equation*}
  as $k \to \infty$. Moreover, there exists a constant $c>0$ with
  $q_\lambda(\tilde w) \ge c  \|\tilde w\|_E^2$ for every $\tilde w \in E_n^\perp$. Hence,
  \begin{equation*}
      \liminf_{k \to \infty}   \sup_{F_n(w_k)^+}\cJ_{\lambda^k}\ge \frac{ct^2}{2} \qquad \text{for every $t>0$,}
  \end{equation*}
  which contradicts (\ref{eq:bounded-action-assumption}). Thus, $w \ne 0$.

  Next, by Lemma~\ref{B-R-Lemma}, choose $R>0$ with $\cJ_\lambda(u) \le 0$ for every
  $u \in F_n(w)^+ \setminus B_R$. For every $u = tw + v \in F_n(w)^+ \cap B_R$, we then have $\|v\|_H\le \|tw +v\|_H \le \|tw + v \|_E \le R$ and therefore, by the strong continuity of $I$ as assumed in $(I1)$, 
  \begin{align*}
    \cJ_\lambda(tw + v) &=\tfrac{t^2}{2}q_\lambda(w)+q_\lambda(v)- I(tw+v)\\
    &\le \liminf_{k \to \infty}\Bigl[ \tfrac{t^2}{2}q_\lambda(w_k)+q_\lambda(v)- I(tw_k+v)\Bigr]\\
    &\le \liminf_{k \to \infty}\Bigl[ \tfrac{t^2}{2}q_{\lambda^k}(w_k)+q_{\lambda^k}(v)- I(tw_k+v)+ |\lambda^k-\lambda|\bigl(\tfrac{t^2}{2}\|w_k\|_H^2 + \|v\|_H^2\bigr)\Bigr]\\
                        &\le \liminf_{k \to \infty}\Bigl[ \cJ_{\lambda^k}(tw_k+v)+ |\lambda^k-\lambda|\bigl(\tfrac{t^2}{2} + R^2\bigr)\Bigr]\\
    &\le \liminf_{k \to \infty} \sup_{F_n(w_k)^+}\cJ_{\lambda^k}.
  \end{align*}
  We thus conclude that
  \begin{equation*}
      \sup_{F_n(w)^+}\cJ_\lambda = \sup_{F_n(w)^+ \cap B_R}\cJ_\lambda \le \liminf_{k \to \infty} \sup_{F_n(w_k)^+}\cJ_{\lambda^k},
  \end{equation*}
  as claimed.
\end{proof}

\begin{lem}
  \label{c-lambda-attained}
The value $c_\lambda$ in (\ref{eq:inf-sup-characterization}) is attained and positive.
\end{lem}

\begin{proof}
  Let $(u_k)_k$ be a minimizing sequence for $\cJ_\lambda$
  in $\cNP_\lambda$. Consider
$\Pi_n: E \to E_n$ the $E$-orthogonal projection onto $E_n$
and define
$w_k := \frac{u_k - \Pi_n u_k}{\|u_k - \Pi_n u_k\|_E}$.
Then, $w_k$ belongs to $E_n^\perp$ and satisfies
$\| w_k \|_E = 1$ and
$F_n(w_k)^+ = F_n(u_k)^+$, so that,
recalling Proposition \ref{prop:criterionB2I},
  \begin{equation*}
   \sup_{F_n(w_k)^+}\cJ_{\lambda}= \sup_{F_n(u_k)^+}\cJ_\lambda = \cJ_\lambda(u_k) \to c_\lambda  \qquad \text{as $k \to \infty$,} 
  \end{equation*}
so \eqref{eq:bounded-action-assumption} holds with $\lambda = \lambda^k$ for all $k$. By Lemma~\ref{general-compactness-cond}, we may therefore pass to a subsequence with the property that $w_k \weak w$ in $E$ as $k \to \infty$ for some
$w \in E_n^\perp \setminus \{ 0 \}$, and 
  $$
  \sup_{F_n(w)^+}\cJ_\lambda  \le \liminf_{k \to \infty}\sup_{F_n(w_k)^+}\cJ_{\lambda}=
  \lim_{k \to \infty}\cJ_{\lambda}(u_k) = c_\lambda.
  $$
  Let $u \in F_n(w)^+$ be the unique maximizer of $\cJ_\lambda$ on $F_n(w)^+$. Then $u \in \cNP_\lambda$, and $\cJ_\lambda(u) \le c_\lambda$. Therefore equality must hold by the definition of $c_\lambda$, and thus $u$ is a minimizer of $\cJ_\lambda$ on $\cNP_\lambda$. The positivity of $c_\lambda$ now follows a posteriori from (\ref{sup-positive}).
\end{proof}

To finish the proof of Theorem~\ref{main-abstract-result-existence-lambda}, it remains to show that every minimizer of $\cJ_\lambda$ on $\cNP_\lambda$ is a critical point of $\cJ_\lambda$. This will be proved with the help of the minimax characterization of $c_\lambda$ in (\ref{eq:inf-sup-characterization}) and the following lemma.

 \begin{lem}
\label{m-lambda-continuous}   
Let $\lambda$ be fixed and,
for every $w \in E \setminus E_n$,
let $u:= m_\lambda(w)$ be the unique global maximizer
of $\cJ_\lambda$ on $F_n(w)^+$.
Then the map $m_\lambda: E \setminus E_n \to \cNP_\lambda$
is continuous.    
 \end{lem}

 \begin{proof}
   Let $w \in E \setminus E_n$, and let $(w_k)_k$ be a sequence in $E \setminus E_n$ with $w_k \to w$. By a standard argument, it suffices to show that
   \begin{equation}
     \label{eq:subsequence-sufficient}
    m_\lambda(w_k) \to m_\lambda(w) \qquad \text{after passing to a subsequence.} 
   \end{equation}
   Without loss of generality, we may assume that
   $w, w_k \in E_n^\perp$, and that $\|w\|_E = 1$,
   $\|w_k\|_E = 1$ for all $k$. We write 
   $$
   u_k:= m_\lambda(w_k) = s_k w_k + v_k
   $$
   with $s_k \in (0,\infty)$ and $v_k \in E_n$. We first claim that
   $u_k$ remains bounded in $E$. If not, we may pass to a subsequence with $\|u_k\|_E \to \infty$. Putting
   $$
   \tilde u_k = \frac{u_k}{\|u_k\|_E}= \frac{s_k}{\|u_k\|_E}w_k + \frac{v_k}{\|u_k\|_E},
   $$
   we note that $\frac{s_k}{\|u_k\|_E}$ remains bounded in $\R$, while $\frac{v_k}{\|u_k\|_E} \in E_n$ remains bounded in $E$. Thus, after passing to a subsequence, 
   there exist $s_* \ge 0$ and $v \in E_n$ with
   $$
   \tilde u_k \to \tilde u := s_* w + v.
   $$
   Moreover, there exists $t>0$ with $\frac{I(t \tilde u)}{t^2} > \frac{q_\lambda(\tilde u)}{2}$.
   Consequently, for $k$ large,
   \begin{align*}
   0 \le   \cJ_\lambda(u_k)&= \|u_k\|_E^2\Bigl(\frac{q_\lambda(\tilde u_k)}{2}- \frac{I(\|u_k\|_E \tilde u_k)}{\|u_k\|_E^2}\Bigr)\\
& \le \|u_k\|_E^2\Bigl(\frac{q_\lambda(\tilde u_k)}{2}- \frac{I(t \tilde u_k)}{t^2}\Bigr)\\     
& = \|u_k\|_E^2\Bigl(\frac{q_\lambda(\tilde u)}{2}- \frac{I(t \tilde u)}{t^2} +o(1)\Bigr)\\     
& \to -\infty.  
   \end{align*}
   This is a contradiction, and thus $(u_k)_k$ remains bounded in $E$. We may thus pass to a subsequence with $u_k \to u:= s_* w + v$, where $s_k \to s_*$ and $v_k \to v$ as $k \to \infty$. We note that $s_* \ne 0$, since otherwise $u \in E_n$ and therefore
   $$
   0 < c_\lambda
   \le \lim_{k \to \infty} \cJ_\lambda(u_k)
   = \cJ_\lambda(u)
   \le 0,
   $$
   a contradiction. Hence $u \in E \setminus E_n$. Moreover, for every $s >0$ and $\tilde v \in E_n$, we have
   \begin{align*}
       \cJ_\lambda(s u + \tilde v)
       = \lim_{k \to \infty} \cJ_\lambda(s u_k + \tilde v)
       \le \lim_{k \to \infty} \cJ_\lambda(u_k)
       = \cJ_\lambda(u).
   \end{align*}
    Consequently, $u$ maximizes $\cJ_\lambda$
    on $F_n(u)^+=F_n(w)^+$, which implies that
   $$
   m_\lambda(w)= u=\lim_{k \to \infty}u_k = \lim_{k \to \infty}m_\lambda(w_k).
   $$
   Hence, \eqref{eq:subsequence-sufficient} holds.
 \end{proof}

 \begin{lem}
\label{minimizer=criticalpoint}   
Let $u \in \cNP_\lambda$ be a minimizer of $\cJ_\lambda$ on $\cNP_\lambda$. Then $u$ is a critical point of $\cJ_\lambda$.
\end{lem}
 
\begin{proof}
Suppose by contradiction that $\cJ_\lambda'(u) \ne 0$.
Then there exist $w \in F_n(u)^\perp$
and $\eps>0$ with $\cJ'(u)w< -\eps$.
Since $\cJ_\lambda$ is of class $C^1$,
this implies that there exists $\delta>0$ with 
\begin{equation}
  \label{eq:delta-eps-continuity-J-prime}
  \cJ_\lambda'(v)w < - \eps \qquad \text{for every $v \in B_\delta(u)$,}
\end{equation}
where $B_\delta(u)$ denotes the open $\delta$-ball in the $E$-norm centered at $u$. We let $w_k:= u+ \frac{1}{k} w$. By Lemma~\ref{m-lambda-continuous}, we then have that
  \begin{equation}
\label{eq:u_k-u-m-lambda}
u_k:= m_\lambda(w_k) \to m_\lambda(u)=u \qquad \text{as $k \to \infty$.}
  \end{equation}
We write
  \begin{equation}
      \label{decomp_u_k}
      u_k = s_k w_k + v_k = s_k u + \frac{s_k}{k} w + v_k
  \end{equation}
  with $s_k\in (0,\infty)$ and $v_k \in E_n$.
  We recall that $u \in F_n(u) \setminus E_n$
  and $w \in F_n(u)^\perp$.
  Projecting \eqref{decomp_u_k} on $F_n(u)$,
  we deduce that $s_k u + v_k$ converges to $u$
  in $F_n(u)$. Since $F_n(u)$ is finite dimensional,
  this implies that $v_k \to 0$
  and $s_k \to 1$ as $k \to \infty$.
  Consequently, for $k$ sufficiently large, we have
  $$
  s_k u + \frac{s s_k}{k}  w+ v_k \in B_\delta(u) \qquad \text{for $s \in (0,1)$,}
  $$
  which together with (\ref{eq:delta-eps-continuity-J-prime}) implies that
$$
  \cJ_\lambda(u)-\cJ_\lambda(u_k) \ge \cJ_\lambda(s_k u+v_k) -\cJ_\lambda(u_k)= - \frac{s_k}{k}
  \int_{0}^1\cJ'\Bigl(s_k u +\frac{ss_k}{k} w+ v_k\Bigr)w
  \intd s
\ge  \frac{s_k}{k} \eps > 0
$$
for $k$ sufficiently large, which contradicts the minimizing property of $u$. 
Hence $u$ is a critical point of $\cJ_\lambda$, as claimed.
\end{proof}

\section{The $\lambda$-dependence
of the minimization problem and normalized solutions}
\label{sec:lambda-depend-minim}

In this section, we study the $\lambda$-dependence
of the Nehari-Pankov solutions. Our main goal is to prove
Theorem~\ref{M-n-char}.

\begin{rem}{\rm 
  \label{sec:lambda-depend-minim-1}
Since $\cJ_\lambda(u) \le \cJ_\mu(u)$ for all $u \in E$ if $\mu \le \lambda$, the inf-sup characterization in (\ref{eq:inf-sup-characterization}) immediately implies that $c_\lambda$ is decreasing in $\lambda \in (\lambda_n,\lambda_{n+1})$.
}
\end{rem}

\subsection{Norms of Nehari-Pankov solutions}
In the following, we let
\begin{equation}
    \label{eq:def-Q-lambda}
    Q_\lambda
    := \{u \in \cNP_\lambda \mid \cJ_\lambda(u)
    = c_\lambda\},
\end{equation}
which is nonempty according
to Theorem~\ref{main-abstract-result-existence-lambda}.

\begin{lem}
\label{c-lambda-continuity-etc}
  The map $\lambda \mapsto c_\lambda$ is continuous in $(\lambda_n, \lambda_{n+1})$. Moreover, we have
  \begin{equation}
    \label{eq:limsup-inf-est}
  \limsup_{\eps \to 0^+}\frac{c_{\lambda-\eps}-c_{\lambda}}{\eps}
  \le \inf \biggl\{\frac{\|u\|_H^2}{2} \Bigm|
  u \in Q_\lambda\biggr\}
      \end{equation}
    and 
  \begin{equation}
    \label{eq:liminf-sup-est}
\sup \biggl\{\frac{\|u\|_H^2}{2}\Bigm| u \in Q_\lambda\biggr\}
\le   \liminf_{\eps \to 0^+}\frac{c_{\lambda}-c_{\lambda+\eps}}{\eps}. 
      \end{equation}
\end{lem}

\begin{proof}
  We first show (\ref{eq:limsup-inf-est}). Let $\delta >0$ and $u \in Q_\lambda$. We claim that there exists $\eps_0>0$ and $\kappa>0$ with
  \begin{equation}
  \label{claim-eps-1}  
  \cJ_{\lambda-\eps}(w) \le c_{\lambda}- \kappa
  \qquad \text{for all $\eps \in [0,\eps_0]$
  and $w \in F_n(u)^+ \setminus B_\delta(u)$,}
  \end{equation}
  where $B_\delta(u)$ is given in the $\|\cdot\|_E$-norm. If this was false, we would have
  $$
  \lim_{k \to \infty} \cJ_{\lambda-\eps_k}(w_k) \ge c_{\lambda}
  $$
  for a sequence $\eps_k \to 0^+$ and a sequence $(w_k)_k$ in $F_n(u)^+ \setminus B_\delta(u)$. By Lemma~\ref{B-R-Lemma} and the positivity of $c_\lambda$, the sequence $(w_k)_k$ is bounded in $E$. Hence, since $F_n(u)^+ \subset F_n(u)$ and the space $F_n(u)$ is finite dimensional, we may pass to a subsequence such that $w_k \to w$ in $E$ for some $w \in \overline{F_n(u)^+} \setminus B_\delta(u)$, which implies that
  $$
  \cJ_\lambda(w)= \lim_{k \to \infty} \cJ_{\lambda-\eps_k}(w_k) \ge c_\lambda>0.
  $$
  Consequently, $w \in F_n(u)^+ \setminus B_\delta(u)$ by (\ref{E_n-nonpositive}), but then the uniqueness stated in Proposition~\ref{prop:criterionB2I} gives that $u=w \in F_n(u)^+ \setminus B_\delta(u)$, a contradiction. Hence (\ref{claim-eps-1}) is true.

  By the monotonicity property stated in Remark~\ref{sec:lambda-depend-minim-1} and Proposition~\ref{prop:criterionB2I}, we now observe that, for $\eps \in (0,\eps_0)$,
    \begin{align*}
    c_\lambda
    &\le c_{\lambda-\eps}
    \le \sup_{F_n(u)^+}\cJ_{\lambda-\eps}
    = \sup_{F_n(u)^+ \cap B_\delta(u)}\cJ_{\lambda-\eps}\\
    &\le c_\lambda
    + \frac{\eps}{2}\sup_{B_\delta(u)}\|\cdot\|_H^2
    \le c_\lambda
    + \frac{\eps}{2}\Bigl(\|u\|_H +\delta\Bigr)^2
    \end{align*}
  and so
  $$
  \frac{c_{\lambda-\eps}-c_{\lambda}}{\eps}
  \le \frac{1}{2}\Bigl(\|u\|_H+\delta\Bigr)^2
  \qquad \text{for all $\eps \in (0,\eps_0)$.}
  $$
  Since $\delta>0$ was chosen arbitrarily, we deduce that
  $$
  \limsup_{\eps \to 0^+}  \frac{c_{\lambda-\eps}-c_{\lambda}}{\eps} \le \frac{1}{2}\|u\|_H^2.
  $$
  Since $u \in Q_\lambda$ was chosen arbitrarily, (\ref{eq:limsup-inf-est}) follows.\\

From (\ref{eq:limsup-inf-est}) and the fact that the map $\lambda \mapsto c_\lambda$ is decreasing, it immediately follows that this map is left continuous. Next, to prove its right continuity, let $\lambda \in (\lambda_{n},\lambda_{n+1})$, and let $(\lambda^k)^k$ be a sequence in $(\lambda,\lambda_{n+1})$ which converges to $\lambda$.
Since the map $\lambda \mapsto c_\lambda$ is decreasing, it follows that
$$
c_\lambda \ge \limsup_{k \to \infty}c_{\lambda^k}.
$$
Let, moreover, $w_k \in E_n^\perp$ be chosen with $\|w_k\|_E=1$ and
$\sup \limits_{F_n(w_k)^+}\cJ_{\lambda^k} = c_{\lambda^k}$.
From Lemma~\ref{general-compactness-cond}, it then follows that
$$
w_k \weak w \in E\qquad \text{after passing to subsequences}
$$
for some $w \in E_n^\perp \setminus \{0\}$, and that
$$
c_\lambda \le \sup_{F_n(w)^+}\cJ_{\lambda} \le \liminf_{k \to \infty}\sup \limits_{F_n(w_k)^+}\cJ_{\lambda^k}= \liminf_{k \to \infty}c_{\lambda^k}.
$$
This shows that $\lim \limits_{k \to \infty}c_{\lambda^k} = c_\lambda$, and we thus obtain the full continuity of the map $\lambda \mapsto c_\lambda$.

Finally, to show (\ref{eq:liminf-sup-est}),
we again let $u \in Q_\lambda$
and $\delta \in \intervaloo{0, \|u\|_H}$.
The $\eps = 0$ case of \eqref{claim-eps-1}, combined with the monotonicity of $\lambda \mapsto \cJ_\lambda(w)$, gives $\cJ_{\lambda+\eps}(w) \le c_\lambda - \kappa$ for all $\eps \ge 0$ and $w \in F_n(u)^+ \setminus B_\delta(u)$, where $B_\delta(u)$ is given in the $\|\cdot\|_E$-norm. Moreover, by the continuity of $\lambda \mapsto c_\lambda$ established above, there exists $\eps_0>0$ with $c_{\lambda+\eps} > c_\lambda - \kappa$ for $\eps \in (0, \eps_0)$. Using Proposition~\ref{prop:criterionB2I}, we then infer that, for $\eps \in (0,\eps_0)$,
  \begin{align*}
c_{\lambda+\eps} &\le \sup_{F_n(u)^+}\cJ_{\lambda+\eps}
              = \sup_{F_n(u)^+ \cap B_\delta(u)}\cJ_{\lambda+\eps}\\
                 &\le \sup_{w \in F_n(u)^+\cap B_\delta(u)}\Bigl(\cJ_\lambda(w)-\frac{\eps}{2}\|w\|_H^2\Bigr)\\
     &\le c_\lambda -\frac{\eps}{2}\inf_{B_\delta(u)}\|w\|_H^2 \le c_\lambda -\frac{\eps}{2}\Bigl(\|u\|_H-\delta \Bigr)^2
  \end{align*}
  and therefore
  $$
\frac{1}{2}\Bigl(\|u\|_H -\delta\Bigr)^2 \le   \frac{c_{\lambda}-c_{\lambda+\eps}}{\eps}  \qquad \text{for $\eps \in (0,\eps_0)$.}
  $$
Taking $\delta \to 0^+$, we deduce that
  $$
\frac{1}{2}\|u\|_H^2 \le \liminf_{\eps \to 0^+}  \frac{c_{\lambda}-c_{\lambda+\eps}}{\eps}.
  $$
  Since $u \in Q_\lambda$ was also chosen arbitrarily, (\ref{eq:liminf-sup-est}) follows.
\end{proof}

\subsection{General action estimates}

To derive an estimate for $c_\lambda$ for $\lambda$ close to $\lambda_{n+1}$, we need the following basic lemma.

\begin{lem}
    \label{consequence-I2-I3}
    For every $\eps>0$ and every compact set $C \subset E \setminus \{0\}$ there exists $\kappa = \kappa(\eps,C)>0$ with
    $$
    I(t u) \ge \kappa t^2
    \qquad \text{for $u \in C$ and $t \ge \eps$.}
    $$
\end{lem}

\begin{proof}
    The claim follows from (I2) and (I3).
    Indeed, suppose that this is not the case. Then there exist a sequence $(u_n)_n$ in $C$ and a sequence of numbers $\eps \le t_n < \infty$ with
    \begin{equation}
        \label{eq:consequence-I2-I3-contra}
        \frac{I(t_n u_n)}{t_n^2} \to 0 \qquad \text{as $n \to \infty$.}
    \end{equation}
    Since $C$ is compact, we may pass to a subsequence with $u_n \to u \in E \setminus \{0\}$. Therefore, it follows from the monotonicity property in (I3) and \eqref{eq:consequence-I2-I3-contra} that
    \begin{equation*}
        0
        = \lim_{n \to \infty}
        \frac{I(t_n u_n)}{t_n^2}
        \ge \lim_{n \to \infty}
        \frac{I(\eps u_n)}{\eps^2}
        = \frac{I(\eps u)}{\eps^2}
    \end{equation*}
    and therefore $I(\eps u) \le 0$,
    which contradicts (I2) since $\eps u \ne 0$.
    The claim follows.
\end{proof}

   \begin{lem}
     \label{flat-initial-curve}
     We have $\frac{c_\lambda}{\lambda_{n+1}-\lambda} \to 0$ as $\lambda \to \lambda_{n+1}$. 
   \end{lem}

    \begin{proof}
      We consider the subspace $W:= \spann\{\zeta_1,\dots,\zeta_{n+1}\} \subset E$ and the compact set
      $$
      C:= \{w \in W \mid |w|_H = 1\}.
      $$
      For every $w \in C$ and every $t\ge 0$, we then have
     $$
     \cJ_\lambda(tw) = \frac{t^2}{2}q_\lambda(w)- I(tw) 
     \le (\lambda_{n+1} -\lambda)\frac{t^2}{2} - I(tw).
     $$
     Now let $\eps>0$ be given. For $t \in (0,\eps]$ and $w \in C$, we then have 
     $$
     \cJ_\lambda(tw) \le (\lambda_{n+1} -\lambda)\frac{t^2}{2} \le (\lambda_{n+1} -\lambda)\frac{\eps^2}{2},
     $$
     while for $t \ge \eps$  and $w \in C$ we have, with $\kappa$ given by Lemma~\ref{consequence-I2-I3}, 
$$
\cJ_\lambda(tw) \le (\lambda_{n+1} -\lambda)\frac{t^2}{2}  - \kappa t^2 \le \Bigl(\frac{\lambda_{n+1} -\lambda}{2}-\kappa\Bigr)t^2
$$
and therefore
$$
\cJ_\lambda(tw) \le 0 \qquad \text{if $\lambda_{n+1}-\lambda \le 2 \kappa$.}
$$
Consequently,
$$
0 \le \sup_{W} \cJ_\lambda \le  (\lambda_{n+1} -\lambda)\frac{\eps^2}{2} \qquad \text{if $\lambda_{n+1}-\lambda \le 2 \kappa$.}
$$
Since $F_n(\zeta_{n+1})^+ \subset W$, it follows from the inf-sup characterization of $c_\lambda$ that 
$$
\limsup_{\lambda \to \lambda_{n+1}} \frac{c_\lambda}{\lambda_{n+1} -\lambda}  \le
\limsup_{\lambda \to \lambda_{n+1}} \frac{1}{\lambda_{n+1} -\lambda}\; \sup_{F_n(\zeta_{n+1})^+} \cJ_\lambda
\le
\limsup_{\lambda \to \lambda_{n+1}} \frac{1}{\lambda_{n+1} -\lambda}\; \sup_{W} \cJ_\lambda\le \frac{\eps^2}{2}.
$$
Since $\eps>0$ was chosen arbitrarily, the claim follows.
\end {proof}

We can now turn to the proof of Theorem~\ref{M-n-char}.
\begin{proof}[Proof of Theorem~\ref{M-n-char}]
Recalling \eqref{eq:g-n-characterization}, we first note that for every $\lambda \in (\lambda_n, \lambda_{n+1})$ and $u \in Q_\lambda$, we have
\begin{equation*}
    \frac{\|u\|_H^2}{2} \le g(\lambda) \le g_n
\end{equation*}
by (\ref{eq:liminf-sup-est}). Hence,
$$
M_n \subseteq (0,g_n].
$$
To prove the first inclusion in (\ref{eq:M-n-char}), we let $m \in (0,g_n)$. Then there exists $ {\lambda'} \in (\lambda_n,\lambda_{n+1})$ with $g({\lambda'}) > m$. We then define the map 
$$
h: [{\lambda'}, \lambda_{n+1}] \to \R, \qquad
h(\lambda) := \left\{
  \begin{aligned}
    &c_\lambda +m (\lambda-\lambda_{n+1}),&& \qquad \lambda< \lambda_{n+1},\\
    &0,&& \qquad \lambda = \lambda_{n+1},
  \end{aligned}
\right.  
$$
which is continuous by Lemmas~\ref{c-lambda-continuity-etc} and \ref{flat-initial-curve}, and therefore attains a global minimum at a point ${\lambda_*} \in [{\lambda'}, \lambda_{n+1}]$. Since, using again Lemma~\ref{flat-initial-curve}, we have
\begin{align*}
  \lim_{\eps \to 0^+}\frac{h(\lambda_{n+1} - \eps)-h(\lambda_{n+1})}{\eps} =-m <0, 
\end{align*}
and also
\begin{align*}
  \limsup_{\eps \to 0^+}\frac{h({\lambda'} + \eps)-h({\lambda'})}{\eps} =
    \limsup_{\eps \to 0^+}\frac{c_{{\lambda'} + \eps}-c_{{\lambda'}}}{\eps} +m= -g({\lambda'})+m <0,
\end{align*}
we deduce that ${\lambda_*} \in ({\lambda'},\lambda_{n+1})$. The minimization property then implies,  by (\ref{eq:limsup-inf-est}) and (\ref{eq:liminf-sup-est}), that 
$$
0 \le \limsup_{\eps \to 0^+}\frac{h({\lambda_*} - \eps)-h({\lambda_*})}{\eps} = \limsup_{\eps \to 0^+}\frac{c_{{\lambda_*}-\eps}-c_{{\lambda_*}}}{\eps} -m
\le \inf
\biggl\{\frac{\|u\|_H^2}{2} \Bigm| u \in Q_{{\lambda_*}}
\biggr\} -m
$$
and
$$
0 \ge \liminf_{\eps \to 0^+}\frac{h(\lambda_{*}) -h({\lambda_*}+\eps)}{\eps} = \liminf_{\eps \to 0^+} \frac{c_{{\lambda_*}}-c_{{\lambda_*}+\eps}}{\eps} -m
\ge \sup
\biggl\{\frac{\|u\|_H^2}{2} \Bigm| u \in Q_{{\lambda_*}}
\biggr\} - m
$$
and so
$$
m \le \inf
\biggl\{\frac{\|u\|_H^2}{2} \Bigm| u \in Q_{{\lambda_*}}
\biggr\}\le 
\sup
\biggl\{\frac{\|u\|_H^2}{2} \Bigm| u \in Q_{{\lambda_*}}
\biggr\} \le m,
$$
which implies that
$$
\frac{\|u\|_H^2}{2} = m \qquad \text{for all $u \in Q_{{\lambda_*}}$.}
$$
In particular, we have $m \in M_n$, and therefore the first inclusion in (\ref{eq:M-n-char}) follows.
\end{proof}

\subsection{An additional action estimate}
\label{sec:an-additional-action}
The purpose of this section is to establish a further upper estimate for $c_\lambda$, which will turn out to be useful in our proof of Theorem~\ref{torus_mult}.
For this we make the following additional assumption.

\begin{itemize}
\item[(I5)] there exists $p>2$ and $c_I > 0$ such that
$I(u) \ge c_I \|u\|_H^p$ for all $u \in E$.
\end{itemize}

\begin{lem}
  \label{upper-bound-action}
  If (I1)--(I5) hold, then we have
  $$
  c_\lambda \le  \frac{p-2}{2 c_I^{\frac{2}{p-2}} p^{\frac{p}{p-2}}} (\lambda_{n+1}-\lambda)^{\frac{p}{p-2}}.
  $$
\end{lem}

\begin{proof}
  Recall that $\zeta_{n+1} \in E_n^\perp$ is an eigenfunction of $A$ associated with the eigenvalue $\lambda_{n+1}$ and $\|\zeta_{n+1}\|_H = 1$.
  By \eqref{eq:inf-sup-characterization}, we have
  $$
  c_\lambda \le \sup_{F_n(\zeta_{n+1})^+}\cJ_\lambda.
  $$
  Moreover, for every
  $u = s \zeta_{n+1} + v \in F_n(\zeta_{n+1})^+
  \subset F_n(\zeta_{n+1})$ with $s >0$
  and $v \in E_n$, we have
  \begin{align*}
  \cJ_\lambda(u)&= \frac{1}{2}q_\lambda(u)-I(u)
  \le \frac{s^2}{2} q_\lambda(\zeta_{n+1}) -c_I\|u\|_H^p \\
  &\le \frac{(\lambda_{n+1}-\lambda)s^2}{2}\|\zeta_{n+1}\|_H^2 -c_I s^p \|\zeta_{n+1}\|_H^p = \frac{(\lambda_{n+1}-\lambda)}{2} s^2
 -c_I s^p.
  \end{align*}
  Since the maximum of the right-hand side for $s \ge 0$
  is given by
$$ \frac{p-2}{2 c_I^{\frac{2}{p-2}} p^{\frac{p}{p-2}}}
(\lambda_{n+1}-\lambda)^{\frac{p}{p-2}},
$$
the claim follows.
\end{proof}

\section{Further estimates in the case
of a metric measure space}
\label{sec:furth-estim-case}

Let us now assume that $H = L^2(\Omega)$
for a metric measure space $(\Omega,d,\mu)$, i.e., $(\Omega,d)$ is a compact metric space and $\mu$ is a
finite Borel measure
on $\Omega$. For matters of convenience, we shall simply write
$\intd x$ in place of $\intd\mu(x)$ and $L^p(\Omega)$ in place of $L^p(\Omega, \intd\mu)$ in the following.
We assume that the operator $A$, the space $E$
and the related notations are given as before,
and that the functional $I \in C^1(E)$ has the form 
\begin{equation}
  \label{eq:special-I-metric-measure-space}
I(u) = \frac{1}{p}\int_{\Omega}r(x)|u|^p \intd x
\end{equation}
with some $p>2$ and some weight function $r \in L^\infty(\Omega)$ with $r \ge r_0 >0$ in $\Omega$. Moreover, we assume that
\begin{equation}
  \label{eq:compact-embedding-assumption}
\text{$E$ is compactly embedded in $L^p(\Omega)$.}   
\end{equation}
The special form of the functional $I$ in (\ref{eq:special-I-metric-measure-space})
and (\ref{eq:compact-embedding-assumption}) then imply that assumptions $(I1)$--$(I4)$ from the introduction are satisfied.
To see this, we note in particular that the function
$t \mapsto F(t):= |t|^p$ satisfies (\ref{eq:szulkin-weth-lemma}). Moreover, since
$$
I(u) \ge \frac{r_0}{p}\|u\|_{L^p(\Omega)}^p
\ge \frac{r_0 \mu(\Omega)^{\frac{2-p}{2}}}{p}
\|u\|_{L^2(\Omega)}^p
= \frac{r_0 \mu(\Omega)^{\frac{2-p}{2}}}{p}
\|u\|_H^p \qquad \text{for $u \in E$,}
$$
condition $(I5)$, introduced in Section~\ref{sec:an-additional-action}, is also satisfied with
\begin{equation}
  \label{eq:c-I-metric-measure}
c_I:= \frac{r_0 \mu(\Omega)^{\frac{2-p}{2}}}{p}.
\end{equation}
For fixed $n \in \N$ satisfying $\lambda_n < \lambda_{n+1}$,
the aim of this section is to give a lower bound
on the maximal $L^2$-norm of least action
solutions corresponding to values
$\lambda \in (\lambda_{n},\lambda_{n+1})$.
To do so, we make the following additional assumption, which is motivated
by our application to the semilinear equation~(\ref{torus-eq}) on the flat $2$-torus.

\begin{enumerate}[label=(I6)]
\item\label{I6} There exist constants $C>0$ with $\|\zeta_k\|_{L^\infty(\Omega)} \le C$ for all $k$, and 
$
\sum \limits_{\lambda_k>0}
\frac{1}{\lambda_k} < \infty.
$
\end{enumerate}

The main result of this section
is the following proposition. 

\begin{prop}
\label{eq-abstract-L2-est}
Let $I$ is given by (\ref{eq:special-I-metric-measure-space}), let assumptions (I1)--(I4) and (I6) be satisfied, and let
\begin{equation}
  \label{eq:def-d-lambda}
d(n,\lambda) :=\sum_{k= n+1}^\infty \frac{1}{\lambda_{k}-\lambda}.
\end{equation}
Then we have, with $C$ given by (I6),  
\begin{equation}
  \label{eq:L2-abstract-est-sup-greater-4}
\sup \Bigl\{\|u\|_{L^2} \bigm| \text{$u \in Q_\lambda$
for some $\lambda \in (\lambda_n,\lambda_{n+1})$}\Bigr \} \ge \sup_{\lambda \in [\lambda_n,\lambda_{n+1})} \Bigl[C^2 C_p (\lambda_{n+1}-\lambda)^{\frac{p(p-4)}{(p-2)^2}}d(n,\lambda) \Bigr]^{-\frac{p-2}{4}},
\end{equation}
in the case $p>4$ with 
\begin{equation}
  \label{eq:def-c-o-p-greater-4}
C_p
:=  \|r\|_{L^\infty(\Omega)} r_0^{-\frac{p(p-4)}{(p-2)^2}} \mu(\Omega)^{\frac{p-4}{p-2}},
\end{equation}
and
\begin{equation}
  \label{eq:L2-abstract-est-sup-smaller-4}
\sup \Bigl\{\|u\|_{L^2} \bigm| \text{$u \in Q_\lambda$ for some $\lambda \in (\lambda_n,\lambda_{n+1})$}\Bigr \} \ge \sup_{\lambda \in [\lambda_n,\lambda_{n+1})} \Bigl[C^2 C_p d(n,\lambda) \Bigr]^{-\frac{1}{p-2}},
\end{equation}
in the case $2 <p \le 4$ with 
\begin{equation}
  \label{eq:def-c-o-p-smaller-4}
C_p := \|r\|_{L^\infty(\Omega)} \mu(\Omega)^{\frac{4-p}{2}}.
\end{equation}
\end{prop}

\begin{proof}
In the following, we let $\lambda \in (\lambda_n,\lambda_{n+1})$. We start by noting that by (I6) we have, for $v = \sum \limits_{k=n+1}^\infty c_k \zeta_k
\in E_n^\perp$,
\begin{align}
  \|v\|_{L^\infty}^2 &\le C^2 \Bigl(\sum_{k=n+1}^\infty |c_k| \Bigr)^2 \le C^2 \Bigl(\sum_{k=n+1}^\infty \frac{1}{\lambda_k-\lambda}\Bigr) \Bigl(\sum_{k=n+1}^\infty (\lambda_k -\lambda)|c_k|^2\Bigr)\nonumber \\
                     &= C^2 \Bigl(\sum_{k=n+1}^\infty \frac{1}{\lambda_k-\lambda}\Bigr)q_\lambda(v)=C^2 d(n,\lambda) q_\lambda(v). \label{v-L-infty-est}
\end{align}
Next we let $u \in Q_\lambda$, so $u$ is a least action solution of (\ref{eq:nonlinear-wave-intro}) satisfying the basic estimate
\begin{equation}
  \label{eq:simple-energy-est}
c_\lambda = \cJ_\lambda(u)= \Bigl(\frac{1}{2}-\frac{1}{p}\Bigr)\int_{\Omega}r(x)|u|^p \intd x
\ge  \frac{r_0(p-2)}{2p} \|u\|_{L^p}^p.
\end{equation}
Next, we derive an estimate for the $L^1$-norm of the function 
$$
V_u: \Omega \to \R, \qquad V_u(x)= r(x)|u(x)|^{p-2},
$$
in terms of the $L^2$-norm of $u$, and we need to distinguish cases for this.
In the case $2<p\le 4$, Hölder's inequality implies that
\begin{equation}
  \label{eq:p-smaller-4-est}
  \|V_u\|_{L^1} \le C_p \|u\|_{L^2}^{p-2}
\end{equation}
with $C_p$ given in (\ref{eq:def-c-o-p-smaller-4}).
In the case $p > 4$, we have (using again Hölder's inequality)
\begin{equation*}
\|V_u\|_{L^1} \le \|r\|_{L^\infty}\|u\|_{L^{p-2}}^{p-2} \le \|r\|_{L^\infty} \|u\|_{L^2}^{\frac{4}{p-2}} \|u\|_{L^p}^{p-\frac{2p}{p-2}}
= \|r\|_{L^\infty} \|u\|_{L^2}^{\frac{4}{p-2}} \|u\|_{L^p}^{\frac{p(p-4)}{p-2}}.
\end{equation*}
Hence, using Lemma~\ref{upper-bound-action}, (\ref{eq:c-I-metric-measure}) and (\ref{eq:simple-energy-est}), we obtain
\begin{align}
\|V_u\|_{L^1}
&\le\|r\|_{L^\infty} \|u\|_{L^2}^{\frac{4}{p-2}} \|u\|_{L^p}^{\frac{p(p-4)}{p-2}} \le \|r\|_{L^\infty} \|u\|_{L^2}^{\frac{4}{p-2}} 
\Bigl(\frac{2p c_\lambda}{r_0(p-2)}\Bigr)^{\frac{p-4}{p-2}}\nonumber\\
&\le \|r\|_{L^\infty} \|u\|_{L^2}^{\frac{4}{p-2}}
\left[\Bigl(\frac{2p }{r_0(p-2)}\Bigr)
\frac{p-2}{2 c_I^{\frac{2}{p-2}} p^{\frac{p}{p-2}}}
(\lambda_{n+1}-\lambda)^{\frac{p}{p-2}}
\right]^{\frac{p-4}{p-2}} \nonumber \\
 &= \|r\|_{L^\infty} \|u\|_{L^2}^{\frac{4}{p-2}}
\biggl(\frac{1}{r_0 (p c_I)^{\frac{2}{p-2}}}
(\lambda_{n+1}-\lambda)^{\frac{p}{p-2}}
\biggr)^{\frac{p-4}{p-2}}
= \|r\|_{L^\infty} \|u\|_{L^2}^{\frac{4}{p-2}}
\biggl(\frac{\mu(\Omega)}{r_0^{\frac{p}{p-2}}}
(\lambda_{n+1}-\lambda)^{\frac{p}{p-2}}
\biggr)^{\frac{p-4}{p-2}} \nonumber \\
&= C_p \|u\|_{L^2}^{\frac{4}{p-2}}(\lambda_{n+1}-\lambda)^{\frac{p(p-4)}{(p-2)^2}} \label{eq:p-grater-4-est-1}
\end{align}
with $C_p$ given in (\ref{eq:def-c-o-p-greater-4}).

\medbreak
Next we observe that, since $u \in Q_\lambda$, the function $u$ is an eigenfunction of the operator $A-V_u$ corresponding to the eigenvalue $\lambda$. Here, we note that 
$A-V_u$ is a self adjoint operator defined as a form sum of the operators $A$ and $V_u$. Here we note that, by Hölder's inequality, $V_u$ is a well-defined and continuous multiplication operator $L^p(\Omega) \to L^{p'}(\Omega)$. Therefore, by (\ref{eq:compact-embedding-assumption}) and duality, it also defines a compact operator $E \to E'$. In other words, $V_u$ is a symmetric and relatively form compact perturbation of the operator $A$, which implies that $A-V_u$ is also bounded from below and has a compact resolvent, see \cite[Chapter XIII.4 and Problem XIII.39]{reed-simon}.
Therefore $A-V_u$ admits a sequence of eigenvalues
$$
\lambda_1(u) \le \lambda_2(u) \le \dots \le \lambda_k(u)
\to \infty \qquad \text{as $k \to \infty$.}
$$
Moreover, since $V_u \ge 0$,
we have $\lambda_k(u) \le \lambda_k$
for every $k \in \N$, which implies that
$\lambda = \lambda_k(u)$ for some $k \ge n+1$.
From the sup-inf characterization
of this eigenvalue, it follows that
\begin{align*}
\lambda &= \sup_{\stackrel{E' \subset E}{\codim E' \le k-1}}\inf_{v \in E' \setminus \{0\}}\frac{1}{\|v\|_{L^2}^2}\Bigl(\langle A^{1/2}v, A^{1/2}v\rangle - \int_{\Omega}V_u v^2\intd x\Bigr)\\ 
&\ge \inf_{v \in E_n^\perp \setminus \{0\}}\frac{1}{\|v\|_{L^2}^2}\Bigl(\langle A^{1/2}v, A^{1/2}v\rangle - \int_{\Omega}V_u v^2\intd x\Bigr)
\end{align*}
and therefore
\begin{equation}
  \label{eq:lambda-characterization}
\inf_{v \in E_n^\perp \setminus \{0\}}\frac{1}{\|v\|_{L^2}^2}\Bigl(q_\lambda(v) - \int_{\Omega}V_u v^2\intd x\Bigr)\le 0.
\end{equation}
Moreover, by (\ref{v-L-infty-est}) and (\ref{eq:p-grater-4-est-1}) we get, in the case $p>4$,
\begin{align*}
  q_\lambda(v)-\int_{\Omega}V_u v^2\intd x &\ge q_\lambda(v)- \|V_u\|_{L^1}
                                          \|v\|_{L^\infty}^2\ge q_\lambda(v)\Bigl(1 - C^2 \|V_u\|_{L^1}d(n,\lambda)\Bigr)\\
  &\ge q_\lambda(v)\Bigl(1 - C^2 C_p \|u\|_{L^2}^{\frac{4}{p-2}}(\lambda_{n+1}-\lambda)^{\frac{p(p-4)}{(p-2)^2}}d(n,\lambda)\Bigr) \qquad \text{for $v \in E_n^\perp$.}
\end{align*}
Combining this estimate with (\ref{eq:lambda-characterization}) gives 
$$
\inf_{v \in E_n^\perp \setminus \{0\}}\frac{q_\lambda(v)}{\|v\|_{L^2}^2} \Bigl(1 - C^2 C_p\|u\|_{L^2}^{\frac{4}{p-2}}(\lambda_{n+1}-\lambda)^{\frac{p(p-4)}{(p-2)^2}}d(n,\lambda)\Bigr) \le 0.
$$
Since
$\frac{q_\lambda(v)}{\|v\|_{L^2}^2} \ge \lambda_{n+1}-\lambda>0$
for every $v \in E_n^\perp \setminus \{ 0 \}$, we conclude in the case $p>4$ that 
$$
1 - C^2 C_p\|u\|_{L^2}^{\frac{4}{p-2}}(\lambda_{n+1}-\lambda)^{\frac{p(p-4)}{(p-2)^2}}d(n,\lambda) \le 0
$$
and therefore
\begin{equation*}
    \|u\|_{L^2}
    \ge \Bigl[
    C^2 C_p (\lambda_{n+1}-\lambda)^{\frac{p(p-4)}{(p-2)^2}}
    d(n,\lambda)\Bigr]^{-\frac{p-2}{4}}.
\end{equation*}
By continuity of the right-hand side of this inequality in $\lambda \in [\lambda_n,\lambda_{n+1})$,
\eqref{eq:L2-abstract-est-sup-greater-4} now follows.

In the case $2<p\le 4$, we use \eqref{v-L-infty-est} and (\ref{eq:p-smaller-4-est}) to obtain similarly that
$$
  q_\lambda(v)-\int_{\Omega}V_u v^2\intd x \ge  q_\lambda(v)\Bigl(1 - C^2 \|V_u\|_{L^1}d(n,\lambda)\Bigr)
  \ge q_\lambda(v)\Bigl(1 - C^2C_p \|u\|_{L^2}^{p-2} d(n,\lambda)\Bigr) \qquad \text{for $v \in E_n^\perp$}
$$
and therefore 
$$
\inf_{v \in E_n^\perp \setminus \{0\}}\frac{q_\lambda(v)}{\|v\|_{L^2}^2}
\Bigl(1 - C^2C_p \|u\|_{L^2}^{p-2} d(n,\lambda)\Bigr) \le 0.
$$
Using again (\ref{eq:lambda-characterization}) and the fact that $\frac{q_\lambda(v)}{\|v\|_{L^2}^2} \ge \lambda_{n+1}-\lambda>0$ for every $v \in E_n^\perp \setminus \{ 0 \}$, we conclude that 
$$
1 - C^2C_p \|u\|_{L^2}^{p-2}d(n,\lambda) \le 0
$$
and therefore
\begin{equation*}
    \|u\|_{L^2}
    \ge \Bigl[
    C^2 C_p d(n,\lambda)\Bigr]^{-\frac{1}{p-2}}.
\end{equation*}
Again this inequality extends to $\lambda \in [\lambda_n,\lambda_{n+1})$ by continuity, and thus (\ref{eq:L2-abstract-est-sup-smaller-4}) follows. 
\end{proof}

\section{Nonlinear Schrödinger equations
    on compact metric graphs}
  \label{sec:nonl-schr-equat}
  
  In this section, we focus on the study
of NLS equations on compact metric graphs.

\subsection{Setup}
We recall that a compact metric graph
is a structure made of a finite number
of edges, identified to compact intervals of $\R$,
joined together at vertices.

Given a compact metric graph $\cG$,
we take $H = L^2(\cG)$.
We consider $A$ to be the operator acting
as $u \mapsto -u''$ on every edge
of the graph and $D(A)$ to be a vector space
of functions on $\cG$, $H^2$ edge by edge,
such that $A$ is self-adjoint on $D(A)$
(see \cite[Theorem 1.4.4]{BerKuc} for a description
of all such vertex conditions).

Then, the space $E$ is made of functions
$H^1$ on each edge with suitable vertex conditions
(see \cite[Theorem 1.4.11]{BerKuc})
and is thus continuously embedded
in the space $L^p(\cG)$
for any $p \in [1, \infty]$.

\medbreak
Let us consider a function $f \in \mathcal{C}(\R, \R)$
satisfying (f1) and (f2) as in the statement
of Theorem~\ref{graph_mult}.
We consider the functional defined
by $I(u) := \int_{\cG} F(u(x)) \intd x$.
Then, $I$ belongs to $C^1(E, \R)$
and one has that
$I'(u)v = \int_{\cG} f(u(x)) v(x) \intd x$.
With this choice of $I$, the abstract equation
$Au = \lambda u + I'(u)$
becomes the nonlinear Schrödinger equation
\begin{equation*}
    -u'' = \lambda u + f(u)
\end{equation*}
coupled with the vertex conditions encoded in $D(A)$.

\medbreak
We will check that $I$ satisfies assumptions (I1) to (I4).
First, let us list a few useful
properties of functions $f$ satisfying (f1).
\begin{prop}
    \label{prof_f}
    Let $f \in \mathcal{C}(\R, \R)$ satisfy (f1). Then:
    \begin{enumerate}
        \item $f(u) > 0$ for $u > 0$
        and $f(u) < 0$ for $u < 0$;
        \item $F$ is decreasing on $\intervaloc{-\infty, 0}$
        and increasing on $\intervalco{0, +\infty}$;
        \item\label{prop_f_3} for every $u \ne 0$, the function
        $t \mapsto \frac{F(tu)}{t^2}$ is nondecreasing
        on $\intervaloo{0, +\infty}$ with
        \begin{equation*}
            \lim_{t \to 0^+} \frac{F(tu)}{t^2} = 0
            \quad
            \text{and}
            \quad
            \lim_{t \to +\infty} \frac{F(tu)}{t^2} = +\infty.
        \end{equation*}
    \end{enumerate}
\end{prop}

\begin{proof}
    The two first points are easily verified.
    For \ref{prop_f_3}, we have that, for $t > 0$,
    \begin{equation*}
        \frac{F(tu)}{t^2}
        = \int_0^{tu} \frac{f(s)}{t^2} \intd s\\
        = \int_0^{u} |r| \frac{f(t r)}{|tr|}   \intd r,
    \end{equation*}
    which is nondecreasing in  $t \in \intervaloo{0, +\infty}$ with
    $\lim \limits_{t \to 0^+} \frac{F(tu)}{t^2} = 0$ because the function $s \mapsto \frac{f(s)}{|s|}$ is increasing on $\R \setminus \{0\}$ with $\lim \limits_{s \to 0}\frac{f(s)}{|s|}= 0$. Moreover, the limit $\lim \limits_{t \to +\infty} \frac{F(tu)}{t^2} = +\infty$ follows from the limits $\lim \limits_{s \to \pm \infty}\frac{f(s)}{|s|} =  \pm \infty$ assumed in (f1).
\end{proof}

\begin{prop}
    The functional $I \in \mathcal{C}(E, \R)$ defined by
    $I(u) := \int_{\cG} F(u(x)) \intd x$
    satisfies the properties (I1) to (I4).
\end{prop}

\begin{proof}
    The hypothesis (I1) holds since $F$
    is continuous and since weak convergence in $E$
    implies strong $L^\infty(\cG)$ convergence.
    Assumption (I2) follows from the sign of $F$,
    (I3) follows from point \ref{prop_f_3}
    of Proposition~\ref{prof_f} and (I4) from Remark~\ref{I4_practice}.
\end{proof}

\subsection{Weyl's law, large spectral gaps
    and large $L^\infty$ norms}
The following spectral property is crucial
to establish Theorem~\ref{graph_mult}.
\begin{prop}
    \label{large_spectral_gaps}
    Let $\cG$ is a compact metric graph
    with a self-adjoint Laplacian $(D(A), A)$.
    Then, $A$ has arbitrarily large spectral gaps,
    i.e. for all $R > 0$ there exists $i \ge 1$
    such that $\lambda_{i+1} - \lambda_i \ge R$.
\end{prop}

\begin{proof}
    This follows from Weyl's law for eigenvalues of $A$,
    valid for arbitrary self-adjoint realizations
    of the Laplacian on a compact metric graph. Indeed, as stated in \cite[Proposition 4.2]{BolEnd}, the eigenvalue counting function
    $$
    K \mapsto N(K) = \# \{ i  \::\: \lambda_i \le K\}
    $$
    satisfies the asymptotic law
    $$
    N(K) \sim \frac{L}{\pi} \sqrt{K} \qquad \text{as $K \to + \infty$,}
    $$
    where $L$ is the total length of the graph. From this estimate, the claim readily follows.
\end{proof}

Exploiting large spectral gaps, we are able
to obtain solutions with a large $L^\infty$ norm.
\begin{prop}
    \label{large_L_infty}
    Let $\cG$ is a compact metric graph
    with a self-adjoint Laplacian $(D(A), A)$
    and $f \in \mathcal{C}(\R, \R)$ satisfying (f1).
    Then, for every
    $\lambda \in \intervaloo{\lambda_n, \lambda_{n+1}}$
    and any Nehari-Pankov solution $u \in Q_{\lambda}$
    (recall \eqref{eq:def-Q-lambda}), we have
    \begin{equation*}
        \| u \|_{L^\infty(\cG)}
        \ge \tilde{g}^{-1}(\lambda_{n+1}-\lambda),
    \end{equation*}
    where
    \begin{equation*}
        \tilde{g}: \intervalco{0, +\infty}
        \to \intervalco{0, +\infty}:
        t \mapsto
        \left\{
            \begin{aligned}
              &\frac{\max(f(t),-f(-t))}{|t|} , &&\qquad \text{$t>0$,}\\
              &0, &&\qquad \text{$t=0$}  
            \end{aligned}
        \right.
    \end{equation*}
    is a strictly increasing bijection.
\end{prop}

\begin{proof}
The fact that $\tilde{g}$ is a strictly increasing bijection follows from (f1). For $u \in E$, we define 
    $$
    V_u \in L^\infty(\cG; \R), \qquad V_u(x)=         \left\{
            \begin{aligned}
              &\frac{f(u(x))}{u(x)}, &&\qquad \text{$u(x)\ne 0$,}\\
              &0, && \qquad \text{$u(x)=0$}  
            \end{aligned}
          \right.
    $$
    Moreover, for $n \in \N$, we let $\lambda_n(u)$
    denote the $n$-th eigenvalue
    of the self-adjoint Schrödinger operator $A-V_u$
    which is also self-adjoint with domain $D(A)$
    and form domain $E$. Hence $\lambda_n(u)$
    is given by the minimax characterization
    \begin{equation}
        \label{eq:minimax-lambda-n-u}
        \lambda_n(u) = \inf_{\stackrel{E' \subset E}{\dim E' \ge n}}\sup_{v \in E' \setminus \{0\}}\frac{1}{\|v\|_H^2}\Bigl(\langle A^{1/2}v, A^{1/2}v\rangle - \int_{\cG}V_u v^2\intd x\Bigr).
    \end{equation}
    Since $\lambda_n(0)=\lambda_n$ and
    $$
    \frac{1}{\|v\|_H^2}\Bigl|\int_{\cG}V_u v^2\intd x\Bigr|
    \le \|V_u\|_{L^\infty(\cG)}
    \qquad \text{for every $u \in E$, $v \in E \setminus \{0\}$,}
    $$
    it follows from (\ref{eq:minimax-lambda-n-u}) that
    $$
    |\lambda_n(u)-\lambda_n|
    \le \|V_u\|_{L^\infty(\cG)}
    \qquad \text{for every $u \in E$.}
    $$
    In the following, we assume that $u \in Q_\lambda$,
    so $u$ is a least action solution
    of (\ref{eq:nonlinear-wave-intro}).
    Then $u$ is an eigenfunction
    of the operator $A-V_u$ corresponding
    to the eigenvalue $\lambda$.
    Moreover, since $V_u \ge 0$ and $\lambda >\lambda_n$,
    it follows from (\ref{eq:minimax-lambda-n-u})
    that $\lambda = \lambda_{k}(u)$ for some $k \ge n+1$,
    and therefore
    \begin{equation*}
        \lambda_{n+1}-\lambda
        \le \lambda_{n+1}-\lambda_{n+1}(u)
        \le \|V_u\|_{L^\infty(\cG)}
        \le \tilde{g}(\| u \|_{L^\infty(\cG)} ).
    \end{equation*}
    Here, the last inequality follows from the pointwise bound $|V_u(x)| \le \tilde g(|u(x)|)$ combined with the monotonicity of $\tilde g$. This proves the claim.
\end{proof}

Now, a very convenient fact arises: under (f1) and (f2),
we are able to obtain an ``$L^\infty$ to $L^2$'' estimate
for \emph{solutions} of the differential problem,
as shown in the following section.
As we will see, our argument operates on individual edges
and is thus completely independent
of the particular self-adjoint
realization of the Laplacian under study.

\subsection{ODE arguments and the $L^\infty$ to $L^2$ estimate}
In this section, we assume that (f1) holds.

We remark that the ODE
\begin{equation}
    \label{graph_ODE}
    -u'' = \lambda u + f(u)
\end{equation}
is the equation of motion in the potential well
$V_{\lambda}(u)
:= F(u) + \frac{\lambda}{2} |u|^2$
since the equation reads $u'' = -V_{\lambda}'(u)$.
Hence, for any solution $u$ of the ODE,
the \emph{ODE energy} of $u$, given here by
$H_u := \tfrac{1}{2} (u')^2 + V_{\lambda}(u)$
is constant with respect to time.
Moreover, all solutions of the ODE exist globally on $\R$ and are periodic.

\medbreak
In the following, given $M > 0$,
we call $u_{\lambda, M}$ the solution of the Cauchy problem
\begin{equation*}
    \begin{cases}
        -u'' = \lambda u + f(u),\\
        u(0) = M, u'(0) = 0.
    \end{cases}
\end{equation*}
Since $\lambda \ge 0$, $V_{\lambda}(u)$ is decreasing
on $\intervaloc{-\infty, 0}$
and increasing on $\intervalco{0, +\infty}$.
Thus, $u_{\lambda, M}$ is up to time translations the unique
solution of the ODE whose maximum value is equal to $M$.
It is $\tau_{\lambda, M}$-periodic,
and we denote by $m_{\lambda, M} > 0$
the opposite of the minimum value of $u_{\lambda, M}$.
We remark that $V_{\lambda}(M) = V_{\lambda}(-m_{\lambda, M})$,
this quantity being equal to the ODE energy of $u_{\lambda, M}$.

\medbreak
We will split our ODE reasoning (which generalizes \cite[Lemma 3.3]{CGJT})
in a series of lemmas.
\begin{lem}
    \label{bound_L2_norm_ODE}
    Let $\lambda \ge 0$. Then, for any $M > 0$, we have
    \begin{equation}
        \label{estimates_L_infinity_norm_ODE}
        \frac{\tau_{\lambda, M}}{8} \min(m_{\lambda, M}, M)^2
        \le \int_0^{\tau_{\lambda, M}} |u_{\lambda, M}(x)|^2 \intd x
        \le \tau_{\lambda, M} \max(m_{\lambda, M}, M)^2.
    \end{equation}
\end{lem}

\begin{proof}
    Let $H = V_{\lambda}(M)$ be the ODE
    energy of $u_{\lambda, M}$.
    Since $V_{\lambda}(u)$ is decreasing
    on $\intervaloc{-\infty, 0}$
    and increasing on $\intervalco{0, +\infty}$,
    we obtain that
    \begin{equation*}
        \tfrac{1}{2} (u_{\lambda, M}'(x))^2
        \begin{cases}
            \le H - V_{\lambda}(M/2)
            &\text{for all } x \in \intervalcc{0, \tau_{\lambda, M}}
            \text{ s.t.\,} M/2 \le u(x),\\
            \ge H - V_{\lambda}(M/2)
            &\text{for all } x \in \intervalcc{0, \tau_{\lambda, M}}
            \text{ s.t.\,} 0 \le u(x) \le M/2,\\
            \le H - V_{\lambda}(-m_{\lambda, M}/2)
            &\text{for all } x \in \intervalcc{0, \tau_{\lambda, M}}
            \text{ s.t.\,} u(x) \le -m_{\lambda, M}/2,\\
            \ge H - V_{\lambda}(-m_{\lambda, M}/2)
            &\text{for all } x \in \intervalcc{0, \tau_{\lambda, M}}
            \text{ s.t.\,} -m_{\lambda, M}/2 \le u(x) \le 0.
        \end{cases}
    \end{equation*}
    Therefore, a particle in the potential
    well always has a smaller
    speed (in absolute value) when going through the region
    $\intervalcc{-m_{\lambda, M}, -m_{\lambda, M}/2}$
    than when going through
    the region $\intervalcc{m_{\lambda, M}/2, 0}$.
    Similarly, it has a smaller
    speed (in absolute value) when going through the region
    $\intervalcc{M/2, M}$
    than when going through
    the region $\intervalcc{0, M/2}$.
    We deduce that
    \begin{equation*}
        |A| \ge \tfrac{1}{2}\tau_{\lambda, M}
        \qquad \text{where }
        A := \bigl\{ x \in \intervalcc{0, \tau_{\lambda, M}}
        \bigm| u_{\lambda, M}(x) \le -m_{\lambda, M}/2
        \text{ or } M/2 \le u(x)
        \bigr\}
    \end{equation*}
    as the particle spends at least half its time in the zone
    where it has a slower speed.
    
    Regarding the inequalities in \eqref{estimates_L_infinity_norm_ODE},
    the upper bound is trivial
    and the lower bound follows from the inequalities
    \begin{equation*}
        \int_0^{\tau_{\lambda, M}} |u_{\lambda, M}(x)|^2 \intd x
        \ge \int_{A} |u_{\lambda, M}(x)|^2 \intd x
        \ge \tfrac{1}{4} \min(m_{\lambda, M}, M)^2 \, |A|
        \ge \frac{\tau_{\lambda, M}}{8} \min(m_{\lambda, M}, M)^2.
        \qedhere
    \end{equation*}
\end{proof}

For the next lemma, we will make use of assumption (f2)
defined before the statement of Theorem~\ref{graph_mult}.

\begin{lem}
    \label{comparison_m_M}
    Let $f \in \mathcal{C}(\R, \R)$ satisfy (f1) and (f2).
    Then, for all $\lambda \ge 0$ and all $M > 0$,
    assuming that one has either $M \ge \kappa_0 s_0$
    or $m_{\lambda, M} \ge \kappa_0 s_0$, then
    \begin{equation*}
        \min(m_{\lambda, M}, M)
        \ge \frac{1}{\kappa_0} \max(m_{\lambda, M}, M).
    \end{equation*}
\end{lem}

\begin{proof}
    We distinguish two cases.
    
    \medbreak
    \noindent \textit{Case 1: $m_{\lambda, M} > M$.}
    \medbreak
    
    In this case, one has
    $m_{\lambda, M} \ge \kappa_0 s_0  > s_0$,
    so that (f2) gives
    \begin{equation*}
        F\Bigl(\frac{m_{\lambda, M}}{\kappa_0}\Bigr) 
        \le F(-m_{\lambda, M})
        \le F( \kappa_0 m_{\lambda, M}).
    \end{equation*}
    Since $\lambda \ge 0$, we thus obtain
    \begin{equation*}
        V_{\lambda}\Bigl(\frac{m_{\lambda, M}}{\kappa_0}\Bigr)
        \le V_{\lambda}(-m_{\lambda, M})
        \le V_{\lambda}(\kappa_0 m_{\lambda, M}).
    \end{equation*}
    Since $V_{\lambda}$ is strictly increasing on $\intervalco{0, +\infty}$
    and that $V_{\lambda}(-m_{\lambda, M}) = V_{\lambda}(M)$,
    we deduce that
    \begin{equation*}
        \frac{m_{\lambda, M}}{\kappa_0}
        \le M \le \kappa_0 m_{\lambda, M}.
    \end{equation*}
    Consequently,
    \begin{equation*}
        \min(m_{\lambda, M}, M) = M
        \ge \frac{m_{\lambda, M}}{\kappa_0}
        = \frac{1}{\kappa_0} \max(m_{\lambda, M}, M).
    \end{equation*}
    
    \medbreak
    \noindent \textit{Case 2: $m_{\lambda, M} \le M$.}
    \medbreak
    
    The reasoning is similar. Now, we have
    $M \ge \kappa_0 s_0$, so that (f2) gives 
    \begin{equation*}
        F\Bigl(-\frac{M}{\kappa_0}\Bigr)
        \le F(M) \le F(-\kappa_0 M),
    \end{equation*}
    hence
    \begin{equation*}
        V_{\lambda}\Bigl(-\frac{M}{\kappa_0}\Bigr)
        \le V_{\lambda}(M) \le V_{\lambda}(-\kappa_0 M).
    \end{equation*}
    Since $V_{\lambda}$ is decreasing
    on $\intervaloc{-\infty, 0}$
    and that $V_{\lambda}(-m_{\lambda, M}) = V_{\lambda}(M)$,
    we deduce that
    \begin{equation*}
        -\kappa_0 M \le -m_{\lambda, M} \le -\frac{M}{\kappa_0}.
    \end{equation*}
    Consequently,
    \begin{equation*}
        \min(m_{\lambda, M}, M)
        = m_{\lambda, M}
        \ge \frac{M}{\kappa_0} = \frac{1}{\kappa_0} \max(m_{\lambda, M}, M). \qedhere
    \end{equation*}
\end{proof}

\begin{lem}
    \label{large_norm_to_small_tau}
    For every $\eps > 0$, there exists $s_1 > 0$
    such that for all $\lambda \ge 0$ and for all
    solutions with $M \ge s_1$ and $m_{\lambda, M} \ge s_1$,
    one has $\tau_{\lambda, M} \le \eps$.
\end{lem}

\begin{proof}
    We call
    $H_M := V_{\lambda}(M) = V_{\lambda}(-m_{\lambda, M})$.
    It is a standard fact (see e.g.\,\cite[p.\,18]{Ar})
    that the period is given by
    \begin{equation}
        \label{expression_tau}
        \tau_{\lambda, M}
        = 2 \int_{-m_{\lambda, M}}^{M}
        \frac{\intd u}{\sqrt{2(H_M - V_{\lambda}(u))}}.
    \end{equation}
    We will show that
    \begin{equation*}
        \int_{0}^{M}
        \frac{\intd u}{\sqrt{V_{\lambda}(M) - V_{\lambda}(u)}} \longrightarrow 0 \qquad \text{as $M \to +\infty$} 
    \end{equation*}
    uniformly in $\lambda \ge 0$.
    
    \bigbreak
    To do so, we take $u = M t$ and obtain
    \begin{align*}
        \int_{0}^{M}
        \frac{\intd u}{\sqrt{V_{\lambda}(M) - V_{\lambda}(u)}}
        &= \int_0^1 \frac{M \intd t}{\sqrt{F(M) - F(Mt) + \frac{M^2 \lambda}{2}(1 - t^2)}}\\
        &\le \int_0^1 \frac{\intd t}{\sqrt{\frac{F(M)}{M^2} - \frac{t^2 F(M)}{M^2}}}= \frac{\pi M}{2 \sqrt{F(M)}}  \longrightarrow 0 \qquad \text{as $M \to +\infty$,}
    \end{align*}
    where we used that $F(Mt) \le t^2 F(M)$
    for $0 \le t \le 1$ and that
    $\frac{F(M)}{M^2} \to +\infty$ as $M \to +\infty$
    (see point \ref{prop_f_3} of Proposition \ref{prof_f}).
    In the same way, we show that
    \begin{equation*}
        \int_{-m}^{0}
        \frac{\intd u}{\sqrt{V_{\lambda}(-m) - V_{\lambda}(u)}}
   \longrightarrow 0 \qquad \text{as $m \to +\infty$}
    \end{equation*}
    uniformly in $\lambda \ge 0$, which ends the proof, recalling \eqref{expression_tau}.
\end{proof}

Using the previous lemmas,
we deduce the following crucial
``$L^\infty$ to $L^2$'' estimate.

\begin{lem}[$L^\infty$ to $L^2$ estimate]
    \label{Linfty_to_L2}
    Let $\cG$ be a compact metric graph
    with a self-adjoint Laplacian $(D(A), A)$
    and $f \in \mathcal{C}(\R, \R)$ satisfy (f1) and (f2).
    Then, there exists $R > 0$ and $\gamma > 0$
    such that, for all $\lambda \ge 0$,
    all solutions of the problem $Au = \lambda u + f(u)$
    on $\cG$ with $\| u \|_{L^\infty(\cG)} \ge R$ satisfy
    \begin{equation*}
        \gamma \|u\|_{L^\infty(\cG)}
        \le \|u\|_{L^2(\cG)}.
    \end{equation*}
\end{lem}

\begin{proof}
    Let $\ell$ be the length of the smallest edge of the graph.
    Let $s_1$ be given by Lemma~\ref{large_norm_to_small_tau}
    with $\eps = \ell/2$.    
    We take $R := \kappa_0 \max(s_0, s_1)$.
    
    Let $\lambda \ge 0$.
    Given a solution $u$, there exists one edge $e$
    of the graph such that
    $\| u \|_{L^\infty(\cG)} = \| u \|_{L^\infty(e)} \ge R$.
    Since $R \ge \kappa_0 s_0$,
    Lemma~\ref{comparison_m_M}, implies that
    \begin{equation}
        \label{eq:comp_m_M}
        \min(m_{\lambda, M}, M)
        \ge \frac{1}{\kappa_0} \max(m_{\lambda, M}, M)
        \ge \frac{1}{\kappa_0} R
    \end{equation}
    for all $M \ge R$.
    Since $\frac{1}{\kappa_0} R \ge s_1$,
    Lemma~\ref{large_norm_to_small_tau}
    implies that $\tau_{\lambda, M} \le \ell/2$
    for all $M \ge R$.
    Hence, the solution $u_{\mid e}$ completes at least two full periods on $e$.
    Let $I \subseteq e$ be a union of $\lfloor |e|/\tau_{\lambda, M} \rfloor$
    consecutive full periods of $u_{\mid e}$, so that
    $|I| = \lfloor |e|/\tau_{\lambda, M} \rfloor \, \tau_{\lambda, M}
    \ge |e| - \tau_{\lambda, M} \ge |e|/2 \ge \ell/2$.
    Applying Lemma~\ref{bound_L2_norm_ODE} to each period, we obtain
    \begin{equation*}
        \frac{\ell}{16} \min(m_{\lambda, M}, M)^2
        \le \frac{|I|}{8} \min(m_{\lambda, M}, M)^2
        \le \int_I |u(x)|^2 \intd x
        \le \int_{\cG} |u(x)|^2 \intd x.
    \end{equation*}
    Using \eqref{eq:comp_m_M}, we obtain
    \begin{equation*}
        \frac{\ell}{16 \kappa_0^2} \max(m_{\lambda, M}, M)^2
        \le \|u\|_{L^2(\cG)}^2.
    \end{equation*}
    Remarking that
    \begin{equation*}
        \max(m_{\lambda, M}, M)
        = \|u\|_{L^\infty(e)}
        = \|u\|_{L^\infty(\cG)}
    \end{equation*}
    and taking $\gamma := \frac{\sqrt{\ell}}{4 \kappa_0}$
    ends the proof.
\end{proof}

\subsection{Sign of solutions in the presence
    of a constant-sign eigenfunction}
  \label{sec:sign-solut-pres}

The following statement is based on a classical
argument, but we include the proof for completeness.

\begin{prop}
    \label{no_pos_sol}
    Let $\cG$ be a compact metric graph
    with a self-adjoint Laplacian $(D(A), A)$.
    Let $f \in \mathcal{C}(\R, \R)$ be a function
    which is positive on $\intervaloo{0, +\infty}$.
    Assume that there exists an eigenfunction $\xi_i$ of $A$,
    associated to the eigenvalue $\lambda_i$,
    which is positive almost everywhere in $\cG$.
    Then, the NLS equation $-u'' = \lambda u + f(u)$
    coupled with vertex conditions of $D(A)$
    has no nontrivial nonnegative solutions
    for $\lambda \ge \lambda_i$.
\end{prop}

\begin{proof}
    Assume that there exist $\lambda \ge \lambda_i$
    and a nontrivial nonnegative solution $u$
    of the problem
    \begin{equation*}
        Au = \lambda u + f(u).
    \end{equation*}
    Taking scalar products with $\xi_i$,
    we obtain
    \begin{equation*}
        (Au, \xi_i)_{L^2}
        = \lambda \int_{\cG} u \xi_i \intd x
        + \int_{\cG} f(u) \xi_i \intd x.
    \end{equation*}
    Since $\xi_i$ is an eigenfunction of $A$, we have that
    \begin{equation*}
        (Au, \xi_i)_{L^2}
        = \lambda_i \int_{\cG} u \xi_i \intd x.
    \end{equation*}
    We obtain
    \begin{equation*}
        (\lambda_i - \lambda) \int_{\cG} u\xi_i \intd x
        = \int_{\cG} f(u) \xi_i \intd x,
    \end{equation*}
    a contradiction since the left-hand side is nonpositive
    while the right-hand side is positive
    since $u \ge 0$ is not identically equal to $0$.
\end{proof}

As already pointed out in the introduction,
the first eigenfunction is positive in $\cG$
in the case of continuity-Kirchhoff conditions.
For a further study of conditions ensuring
positivity of the first eigenfunction,
one may refer to \cite[Section 4.5]{Kur}.

\subsection{Proof of Theorem~\ref{graph_mult}}
We can finally turn to the proof of our main result
about NLS equations on metric graphs.

\begin{proof}[Proof of Theorem~\ref{graph_mult}]
    Combining Propositions
    \ref{large_spectral_gaps} and \ref{large_L_infty},
    we deduce the existence in bigger and bigger
    spectral gaps of solutions
    whose $L^\infty$-norms converge to $+\infty$.
    Then, using Lemma~\ref{Linfty_to_L2},
    we deduce that bigger and bigger
    spectral gaps contain solutions of $L^2$ norms
    converging to $+\infty$. Since the set of norms
    obtained in each spectral gap is an interval according to
    Theorem~\ref{M-n-char}, we deduce the existence part of Theorem~\ref{graph_mult}.
    
    Finally, the claim about the sign of solutions
    when $A$ admits an eigenvalue positive almost everywhere
    follows from Proposition~\ref{no_pos_sol}
    when $\lambda$ is large.
\end{proof}

\section{The nonlinear biharmonic equation on the $2$-torus}
\label{sec:nonl-biharm-equat}

In this section we complete the proof of Theorem~\ref{torus_mult}.
Namely, we consider the special case
where $\Omega= S^1 \times S^1$ is the flat torus,
$A= \Delta^2$ and the functional $I$
is given by (\ref{eq:special-I-metric-measure-space})
for some $p>2$. It is then easy to see
that assumptions (I1)--(I5) are satisfied,
as the space $E=D(A^{\frac{1}{2}})=D(\Delta)$
coincides with the Sobolev space $H^2(\Omega)$
which is compactly embedded
in $L^\infty(\Omega) \subset L^p(\Omega)$. 

Moreover, in this case, the spectrum of $A$ is given by
$$
\{(k_1^2+k_2^2)^2 \mid (k_1, k_2) \in \Z^2\}.
$$
Indeed, defining
\begin{equation*}
    \Gamma_k := \begin{cases}
        \frac{1}{\sqrt{2\pi}} & \text{if } k = 0, \\
        \frac{1}{\sqrt{\pi}}  & \text{if } k \neq 0,
    \end{cases}
\end{equation*}
then an $L^2(\Omega)$-orthonormal basis
of eigenfunctions of $A$ is given by
$(\psi_{(k_1, k_2)})_{(k_1, k_2) \in \Z^2}$, where
\begin{align*}
    \psi_{(k_1, k_2)}(x_1, x_2)
    &:= \Gamma_{k_1} \Gamma_{k_2} \begin{cases}
        \sin(k_1 x_1) \sin(k_2 x_2)
        &\text{if $k_1 \ge 1$ and $k_2 \ge 1$},\\
        \cos(k_1 x_1) \sin(k_2 x_2)
        &\text{if $k_1 \le 0$ and $k_2 \ge 1$},\\
        \sin(k_1 x_1) \cos(k_2 x_2)
        &\text{if $k_1 \ge 1$ and $k_2 \le 0$},\\
        \cos(k_1 x_1) \cos(k_2 x_2)
        &\text{if $k_1 \le 0$ and $k_2 \le 0$}.
    \end{cases}
\end{align*}
Here, $x_1$ and $x_2$ denote angular variables in $S^1$.
Thus, the dimension of the eigenspace associated
with an eigenvalue $\mu$ of $A$
is given by $r_2(\sqrt{\mu})$, where,
here and in the following,
$$
r_2:
\Z^{\ge 0} \to \Z^{\ge 0}:
j \mapsto r_2(j)
:= \# \bigl\{
    (k_1, k_2) \in \Z^2 \mid j = k_1^2 + k_2^2
\bigr\}
$$
denotes the sums of two squares function. A result from analytic
number theory (see \cite[Chapter 16.9]{hardy-wright}) implies that
$$
r_2(j) \le 4 d(j) \qquad \text{for every $j \in \N$,}
$$
where $d(j)$ is the number of divisors of $j$.
From this inequality, one can deduce (see \cite[Theorem 315]{hardy-wright})
that for every $\eps>0$ there exists $C_\eps>0$ with 
\begin{equation}
  \label{eq:key-asymptotic-est-sums-of-squares}
r_2(j) \le C_\eps j^\eps \quad \text{for every $j \in \N$.}
\end{equation}
Since $\| \psi_{(k_1, k_2)} \|_{L^\infty(\Omega)} \le \frac{1}{\pi}$
for all $(k_1, k_2) \in \Z^2$, the first part of
assumption \ref{I6} holds with $C = \frac{1}{\pi}$. Moreover, as
\begin{equation*}
    \sum_{k \in \Z^{\ge 2}}
    \frac{1}{\lambda_k}
    = \sum_{j \in \Z^{> 0}}
    \frac{r_2(j)}{j^2}
    \le C_{1/2} \sum_{j \in \Z^{> 0}}
    \frac{1}{j^{3/2}}
    < +\infty,
\end{equation*}
we deduce that the hypothesis \ref{I6} is satisfied.

\medbreak
In the following, for fixed $\ell \in \N$,
we choose $n_\ell \in \N$ such that 
$\ell^4 = \lambda_{n_\ell+1}$
is the $(n_\ell+1)$-th eigenvalue of $A = \Delta^2$
(counted with multiplicity).
For every eigenvalue $\mu= (k_1^2+k_2^2)^2< \ell^4$,
we then have $k_1^2 + k_2^2 < \ell^2$.
Therefore, $k_1^2 + k_2^2 \le \ell^2-1$,
so $\mu \le (\ell^2-1)^2$. Hence it follows that
$$
\lambda_{n_\ell} \le (\ell^2-1)^2
$$
so the value
$$
\lambda = \lambda(\ell) := (\ell^2-1)^2
$$
is admissible to get lower bounds in (\ref{eq:L2-abstract-est-sup-greater-4}) and (\ref{eq:L2-abstract-est-sup-smaller-4}). We note that 
\begin{equation}
  \label{eq:ell-diff-est}
\lambda_{n_\ell+1}-\lambda(\ell) = \ell^4-(\ell^2-1)^2 = 2\ell^2+1.
\end{equation}
Moreover, fixing $\eps \in (0,\frac{1}{2})$
in (\ref{eq:key-asymptotic-est-sums-of-squares}),
we can estimate the quantity $d(n_\ell,\lambda(\ell))$
in (\ref{eq:def-d-lambda}) by
\begin{align*}
d(n_\ell,\lambda(\ell))
&= \sum_{\stackrel{(k_1, k_2) \in \Z^2}{(k_1^2+k_2^2)^2
\ge \ell^4}}\frac{1}{(k_1^2+k_2^2)^2-\lambda(\ell)}
= \sum_{j = \ell^2}^\infty \frac{r_2(j)}{j^2-\lambda(\ell)}\\
&\le C_\eps \sum_{j = \ell^2}^\infty
\frac{j^\eps}{j^2-\lambda(\ell)}
= C_\eps \Bigl(\frac{\ell^{2\eps}}{\ell^4-\lambda(\ell)}
+ \sum_{j = \ell^2+1}^\infty
\frac{j^\eps}{j^2-\lambda(\ell)}\Bigr),
\end{align*}
where
$$
\frac{\ell^{2\eps}}{\ell^4-\lambda(\ell)} = \frac{\ell^{2\eps}}{2\ell^2+1} \le (2\ell^2 + 1)^{\eps-1} = (\lambda_{n_\ell+1}-\lambda(\ell))^{\eps-1}
$$
and 
\begin{align*}
\sum_{j = \ell^2+1}^\infty \frac{j^\eps}{j^2-\lambda(\ell)}
&\le \int_{\ell^2}^{+\infty} \frac{t^\eps}{t^2-\lambda(\ell)}
\intd t
= \frac{1}{2}\int_{\ell^4}^{+\infty}
\frac{t^{\frac{\eps-1}{2}}}{t-\lambda(\ell)}\intd t
\le  \frac{\ell^{2(2\eps-1)}}{2}\int_{\ell^4}^{+\infty}
\frac{t^{-\frac{\eps}{2}}}{t-\lambda(\ell)}\intd t\\
&\le  \frac{\ell^{2(2\eps-1)}}{2}
\int_{\ell^4}^{+\infty} (t-\lambda(\ell))^{-\frac{\eps}{2}-1}\intd t
\le \frac{\ell^{2(2\eps-1)}}{\eps}
(\ell^4-\lambda(\ell))^{-\frac{\eps}{2}}
\le \frac{\ell^{2(2\eps-1)}}{\eps}
(2\ell^2+1)^{-\frac{\eps}{2}} \\
&\le \frac{\ell^{-2}}{\eps} (2\ell^2+1)^{\frac{3\eps}{2}}
\le \frac{3}{\eps}(2\ell^2+1)^{\frac{3\eps}{2}-1}
= \frac{3}{\eps}
(\lambda_{n_\ell+1}-\lambda(\ell))^{\frac{3}{2}\eps-1}.
\end{align*}
Combining these estimates, we find that for every $\tau \in (0,1)$ there exists $C_\tau>0$ independent of $\ell$, with
\begin{equation}
  \label{eq:d-n-ell-est}
d(n_\ell,\lambda(\ell))  \le C_\tau (\lambda_{n_\ell+1}-\lambda(\ell))^{-\tau}.
\end{equation}
Inserting this estimate in (\ref{eq:L2-abstract-est-sup-greater-4}) and using (\ref{eq:ell-diff-est}), we then deduce from Proposition~\ref{eq-abstract-L2-est} in the case $p>4$ that 
\begin{align*}
\sup \Bigl\{\|u\|_{L^2} \bigm|
\text{$u \in Q_\lambda$ for some
$\lambda \in (\lambda_{n_\ell},
\lambda_{n_{\ell}+1})$}\Bigr \}
&\ge \Bigl[C^2C_p (\lambda_{n_\ell+1}-\lambda(\ell))^{\frac{p(p-4)}{(p-2)^2}}
d(n_\ell,\lambda(\ell)) \Bigr]^{-\frac{p-2}{4}}\\
&\ge \widetilde C_\tau (\lambda_{n_\ell+1}-\lambda(\ell))^{\bigl(\tau - \frac{p(p-4)}{(p-2)^2}\bigr) \frac{p-2}{4}}\\
&= \widetilde C_\tau (2\ell^2+1)^{\bigl(
\tau - \frac{p(p-4)}{(p-2)^2}\bigr) \frac{p-2}{4}}
\end{align*}
with a constant $\widetilde C_\tau$ which depends on $\tau \in (0,1)$ but not on $\ell$. Since $\frac{p(p-4)}{(p-2)^2}< 1$, we may fix $\tau \in (0,1)$ with
$\tau-\frac{p(p-4)}{(p-2)^2}>0$.
This implies that, for every $\mu>0$,
there exists $\ell_\mu \in \N$ with
\begin{equation}
  \label{eq:key-final-est}
\sup \Bigl\{\|u\|_{L^2} \bigm| \text{$u \in Q_\lambda$ for some $\lambda \in (\lambda_{n_\ell},\lambda_{n_{\ell}+1})$}\Bigr \} \ge \mu +1 \qquad \text{for $\ell \ge \ell_\mu$.}
\end{equation}
In the case $2 <p \le 4$, we come to the same conclusion by inserting the estimate (\ref{eq:d-n-ell-est}) in (\ref{eq:L2-abstract-est-sup-smaller-4}) and using (\ref{eq:ell-diff-est}) again to see that 
\begin{align*}
\sup \Bigl\{\|u\|_{L^2} \bigm|
\text{$u \in Q_\lambda$ for some
$\lambda \in (\lambda_{n_\ell},
\lambda_{n_{\ell}+1})$}\Bigr \}
&\ge \Bigl[C^2C_p d(n_\ell,\lambda(\ell)) \Bigr]^{-\frac{1}{p-2}}\\
&\ge \Bigl[C^2C_p C_\tau]^{-\frac{1}{p-2}}  (\lambda_{n_\ell+1}-\lambda(\ell))^{\frac{\tau}{p-2}}\\
&= \Bigl[C^2C_p C_\tau]^{-\frac{1}{p-2}}(2\ell^2+1)^{\frac{\tau}{p-2}}.
\end{align*}
Combining the estimate (\ref{eq:key-final-est}) with Theorem~\ref{M-n-char} yields the existence of infinitely many normalized solutions of the equation (\ref{torus-eq}) with mass $\mu$ corresponding to an infinite set of parameters $\lambda$ located in the disjoint intervals $(\lambda_{n_\ell},\lambda_{n_{\ell}+1})$, $\ell \ge \ell_\mu$. Moreover, all these solutions are sign changing. Indeed, this follows as in Section~\ref{sec:sign-solut-pres} (or by simply integrating (\ref{torus-eq}) over $\Omega = S^1 \times S^1$), since the eigenspace of $A=\Delta^2$ associated with $\lambda_1=0$ is spanned by constant functions (which therefore do not change sign).
   The proof of Theorem~\ref{torus_mult} is thus finished.

\bigbreak

\section*{Acknowledgements}
The authors thank Louis Jeanjean for valuable comments. D.G. is a Francqui Fellow of the Belgian American Educational Foundation at Brown University
and was a Research Fellow of the F.R.S.-FNRS when this research was initiated.
He would like to thank the Francqui Foundation, the Belgian American Educational Foundation,
the University of Mons, the F.R.S.-FNRS and the Goethe-Universität Frankfurt for supporting this research.

\section*{Statements and declarations}
The authors have no relevant financial
or non-financial interests to disclose.

Data sharing is not applicable to this article
as no datasets were generated
or analyzed during the current study.


\begin{thebibliography}{90}
    \bibitem{Ar}
    V.\,I.~Arnold, \emph{Mathematical Methods of Classical Mechanics}, 2nd ed., Graduate Texts in Mathematics, Springer New York, 1989.
    
    \bibitem{bartsch-mederski-2015}
    T.~Bartsch and J.~Mederski, Ground and bound state solutions of semilinear time-harmonic Maxwell equations in a bounded domain, \emph{Arch.\ Ration.\ Mech.\ Anal.} \textbf{215}(1), 283--306, 2015.
    
    \bibitem{bartsch-mederski-2017-JFA}
    T.~Bartsch and J.~Mederski, Nonlinear time-harmonic Maxwell equations in an anisotropic bounded medium, \emph{J.\ Funct.\ Anal.} \textbf{272}(10), 4304--4333, 2017.
    
    \bibitem{bartsch-mederski-2017-JFPT}
    T.~Bartsch and J.~Mederski, Nonlinear time-harmonic Maxwell equations in domains, \emph{J.\ Fixed Point Theory Appl.} \textbf{19}(1), 959--986, 2017.
    
    \bibitem{BerKuc}
    G.~Berkolaiko and P.~Kuchment, \emph{Introduction to quantum graphs}, American Mathematical Society, 2013.
    
    \bibitem{BolEnd}
    J.~Bolte and S.~Endres, The Trace Formula for Quantum Graphs with General Self Adjoint Boundary Conditions, \emph{Ann.\ Henri Poincaré} \textbf{10}(1), 189--223, 2009.
    
    \bibitem{BCJS_nonlin}
    J.~Borthwick, X.~Chang, L.~Jeanjean, and N.~Soave, Normalized solutions of $L^2$-supercritical NLS equations on noncompact metric graphs with localized nonlinearities, \emph{Nonlinearity} \textbf{36}(7), 3776, 2023.
    
    \bibitem{BCJS_TAMS}
    J.~Borthwick, X.~Chang, L.~Jeanjean, and N.~Soave, Bounded Palais-Smale sequences with Morse type information for some constrained functionals, \emph{Trans.\ Amer.\ Math.\ Soc.} \textbf{377}(6), 4481--4517, 2024.
    
    \bibitem{CGJT}
    P.~Carrillo, D.~Galant, L.~Jeanjean, and C.~Troestler, Infinitely many normalized solutions of $L^2$-supercritical NLS equations on noncompact metric graphs with localized nonlinearities, \emph{Discrete Contin.\ Dyn.\ Syst.} \textbf{45}, 2025.
    
    \bibitem{CJS}
    X.~Chang, L.~Jeanjean, and N.~Soave, Normalized solutions of $L^2$-supercritical NLS equations on compact metric graphs, \emph{Ann.\ Inst.\ H.\ Poincaré C Anal.\ Non Linéaire} \textbf{41}(4), 933--959, 2023.
    
    \bibitem{de-Paiva-Kryszewski-Szulkin-2017}
    F.\,O.~de Paiva, W.~Kryszewski, and A.~Szulkin, Generalized Nehari manifold and semilinear Schrödinger equation with weak monotonicity condition on the nonlinear term, \emph{Proc.\ Amer.\ Math.\ Soc.} \textbf{145}(11), 4783--4794, 2017.
    
    \bibitem{DDGS}
    C.~De Coster, S.~Dovetta, D.~Galant, and E.~Serra, An action approach to nodal and least energy normalized solutions for nonlinear Schrödinger equations, \emph{Ann.\ Inst.\ H.\ Poincaré C Anal.\ Non Linéaire}, Art.\ 160, 2025.
    
    \bibitem{Dov18}
    S.~Dovetta, Existence of infinitely many stationary solutions of the $L^2$-subcritical and critical NLSE on compact metric graphs, \emph{J.\ Differential Equations} \textbf{264}(7), 4806--4821, 2018.
    
    \bibitem{hardy-wright}
    G.\,H.~Hardy and E.\,M.~Wright, \emph{An Introduction to the Theory of Numbers}, 4th ed., Clarendon Press Oxford, 1960.
    
    \bibitem{JeSo}
    L.~Jeanjean and L.~Song, Sign-changing prescribed mass solutions for $L^2$-supercritical NLS on compact metric graphs, \emph{J.\ Funct.\ Anal.} \textbf{290}(10), Art.\ 111422, 2026.
    
    \bibitem{KNP}
    A.~Kairzhan, D.~Noja, and D.\,E.~Pelinovsky, Standing waves on quantum graphs, \emph{J.\ Phys.\ A:\ Math.\ Theor.} \textbf{55}(24), Art.\ 244012, 2022.
    
    \bibitem{Kur}
    P.~Kurasov, \emph{Spectral Geometry of Graphs}, Operator Theory: Advances and Applications, Birkhäuser Berlin, Heidelberg, 2024.
    
    \bibitem{mandel-weth-2025}
    R.~Mandel and T.~Weth, Ground state solutions to generalized nonlinear wave equations with infinite-dimensional kernel, \emph{arXiv preprint arXiv:2510.22544}, 2025.
    
    \bibitem{pankov-2005}
    A.~Pankov, Periodic nonlinear Schrödinger equation with application to photonic crystals, \emph{Milan J.\ Math.} \textbf{73}(1), 259--287, 2005.
    
    \bibitem{pierotti-verzini}
    D.~Pierotti and G.~Verzini, Normalized bound states for the nonlinear Schrödinger equation in bounded domains, \emph{Calc.\ Var.\ Partial Differential Equations} \textbf{56}(5), Art.\ 133, 2017.
    
    \bibitem{ramos-tavares-2008}
    M.~Ramos and H.~Tavares, Solutions with multiple spike patterns for an elliptic system, \emph{Calc.\ Var.\ Partial Differential Equations} \textbf{31}(1), 1--25, 2008.
    
    \bibitem{reed-simon}
    M.~Reed and B.~Simon, \emph{Methods of Modern Mathematical Physics IV: Analysis of Operators}, Academic Press, 1978.
    
    \bibitem{szulkin-weth-2009}
    A.~Szulkin and T.~Weth, Ground state solutions for some indefinite variational problems, \emph{J.\ Funct.\ Anal.} \textbf{257}(12), 3802--3822, 2009.
    
    \bibitem{SzuWet}
    A.~Szulkin and T.~Weth, The Method of Nehari Manifold, in \emph{Handbook of Nonconvex Analysis and Applications} (D.\,Y.~Gao and D.~Motreanu, eds.), International Press Somerville, 597--632, 2010.
\end{thebibliography}
\end{document}